\documentclass[11pt,reqno]{article} 

\usepackage{graphicx}
\graphicspath{ {./images/} }
\usepackage{lipsum}
\usepackage{subcaption}

\usepackage{xcolor}

\usepackage{amsthm}
\usepackage{bbm}
\usepackage{mathrsfs}
\usepackage{dynkin-diagrams}

\usepackage[utf8]{inputenc}

\usepackage{tikz-cd}

\usepackage{enumitem}

\usepackage{amssymb}
\usepackage{float}
\usepackage{amsbsy}
\usepackage[all]{xy}\usepackage{xypic}
\usepackage{braket}
\usepackage[colorlinks = true,citecolor=black]{hyperref}
\usepackage{amsmath}

\usepackage{wrapfig}

\usepackage{authblk}

\usepackage[upint]{stix}

\newcommand\myeq{\mathrel{\stackrel{\makebox[0pt]{\mbox{\normalfont\tiny loc}}}{=}}}

\hypersetup{
     colorlinks=true,
     linkcolor=blue,
     filecolor=blue,
     citecolor = blue,      
     urlcolor=cyan,
     }

\makeatletter
\newcommand\opteq[1]{\mathrel{\mathpalette\opt@eq{#1}}}
\newcommand{\opt@eq}[2]{%
  \begingroup
  \sbox\z@{$#1#2$}%
  \sbox\tw@{\resizebox{!}{.5\ht\z@}{$\m@th#1($}}%
  \nonscript\hskip-\wd\tw@
  \mkern1mu
  \raisebox{-.35\ht\z@}[0pt][0pt]{\resizebox{!}{.5\ht\z@}{$\m@th#1($}}%
  \mkern-1mu
  {#2}%
  \mkern-1mu
  \raisebox{-.35\ht\z@}[0pt][0pt]{\resizebox{!}{.5\ht\z@}{$\m@th#1)$}}%
  \mkern1mu
  \nonscript\hskip-\wd\tw@
  \endgroup
}
\makeatother

\def\XXint#1#2#3{{\setbox0=\hbox{$#1{#2#3}{\int}$}
     \vcenter{\hbox{$#2#3$}}\kern-.5\wd0}}

\newtheorem{theorem}{Theorem}[section]
\newtheorem*{theorem*}{Theorem}
\newtheorem{theorem-non}{Theorem}
\newtheorem{lemma-non}{Lemma}

\theoremstyle{definition} 
\newtheorem{thm}{Theorem}

\theoremstyle{definition}

\newtheorem{conjecture-non}{Conjecture}

\newtheorem{corollary-non}{Corollary}
\newtheorem{proposition}[theorem]{Proposition}
\newtheorem{lemma}[theorem]{Lemma}
\newtheorem*{lemma*}{Lemma}

\newtheorem*{conjecture*}{Conjecture}

\newtheorem{problem}[theorem]{Problem}
\theoremstyle{definition}
\newtheorem{definition}[theorem]{Definition}

\theoremstyle{remark}
\newtheorem{remark}[theorem]{Remark}

\DeclareMathOperator{\rank}{rank}

\numberwithin{equation}{section}

\usepackage{geometry}
\begin{document}
\title{{\bf{Weitzenböck Remainder Spectrum on Rational Homogeneous Varieties}}}

\author[1]{Eder M. Correa}
\author[2]{Lucas Almeida}
\author[3]{Samuel Wainer}
\affil[1,2]{Universidade Estadual de Campinas, Brazil\\

Instituto de Matemática, Estatística e Computação Científica}
\affil[1]{ederc@unicamp.br}
\affil[3]{ Instituto Tecnológico de Aeronáutica, Brazil}


\maketitle

\begin{abstract}
   In this paper, we precisely describe the spectrum of closed invariant $(1,1)$-forms viewed as an operator acting on complex spinor bundles over rational homogeneous varieties. Using this result, we describe the spectrum of the Weitzenböck remainder of ${\rm{Spin}}^{c}$ Dirac operators on rational homogeneous varieties. In particular, we present an explicit formula for their smallest eigenvalue. As a byproduct, we obtain a new lower bound for the eigenvalues of the ${\rm{Spin}}^{c}$ Dirac operator, expressed in terms of Lie-theoretic data. Additionally, combining the Atiyah-Singer index theorem with the Borel-Weil-Bott theorem, we provide a complete classification of ${\rm{Spin}}^{c}$ structures on rational homogeneous varieties which admit harmonic spinors. In this last setting, we present an explicit formula for the index of the associated ${\rm{Spin}}^{c}$ Dirac operator in terms of Lie theory. 
\end{abstract}

\hypersetup{linkcolor=black}
\tableofcontents

\hypersetup{linkcolor=black}

\section{Introduction}

A distinguished feature that makes ${\rm{Spin}}^{c}$ structures so widely studied in complex geometry is that they allow certain complex differential operators (e.g. the Dolbeault operator \cite{hitchin1974harmonic}) to be treated as Dirac operators, unlocking the use of analytical methods from differential geometry, such as the Bochner technique, see for instance \cite{duistermaat1996heat}, \cite{nicolaescu2000notes}, and \cite{roe2013elliptic}.

It was shown by K-D. Kirchberg \cite{kirchberg1986estimation} that on a compact Kähler manifold $(M,g,J)$ the first eigenvalue $\lambda$ of the Dirac operator satisfies
\begin{equation}
\label{eq2}
\lambda^{2} \geq
    \begin{cases}
        \displaystyle \frac{m+1}{4m}\inf_{M}S_{g}, \ \text{if} \ m \ \text{is odd}, \\
        \\
        \displaystyle \frac{m}{4(m-1)}\inf_{M}S_{g}, \ \textrm{if m is even, }
        \end{cases}
\end{equation}
where $S_{g}$ is the scalar curvature of $M$ and $m = \dim_{\mathbbm{C}}(M)$. In the above setting, Kirchberg’s estimates rely essentially on the decomposition of the spinor bundle $\Sigma M$ under the action of the K\"{a}hler form. Kirchberg was the first to adapt Friedrich's techniques \cite{friedrich1980erste} specifically for Kähler manifolds, proving a much sharper bound that depends directly on the complex dimension of the manifold.

In 1994, Oussama Hijazi published a significant work on the eigenvalues of the Dirac operator on compact Kähler manifolds \cite{hijazi_eigenvalues_kahler}. In this study, by modifying the standard twistor operator into a specific Kählerian twistor operator, he provided a much cleaner proof of the bounds and completely characterized the geometric conditions of the limiting spinors. In the same year, Witten introduced in \cite{witten1994monopoles} the Seiberg-Witten equations, showing that by coupling the ${\rm{Spin}}^{c}$ Dirac operator to a gauge field, one could classify smooth 4-manifolds and complex surfaces. 

Following the Seiberg-Witten revolution, differential geometers began actively extending the eigenvalue estimates of Friedrich, Kirchberg, and Hijazi from standard {\rm{Spin}} manifolds to general ${\rm{Spin}}^{c}$ manifolds, e.g. \cite{herzlich1999generalized} and \cite{kim2000einstein}. The core mathematical challenge introduced in these works is that the ${\rm{Spin}}^{c}$ bounds must account not just for the scalar curvature of the manifold, but also for the curvature of the twisting complex line bundle. 

Recently, Nakad and Pilca in \cite{nakad2015eigenvalue} derived an estimate for the eigenvalues of the ${\rm{Spin}}^{c}$ Dirac operator, by restricting the attention to compact Kähler-Einstein manifolds with a specific scalar curvature and endowed with the particular class of ${\rm{Spin}}^c$ structures provided by suitable powers of the irreducible root of the canonical line bundle of $(M,g,J)$. More precisely, they obtained the following estimate
\begin{equation}
\label{eqlb}
    \lambda^{2}  \geq \left( 1 - \frac{q^{2}}{p^{2}} \right) (m+1)^{2}, 
\end{equation}
where $\lambda$ is an eigenvalue of the Dirac operator $\mathcal{D}_{A}$, $m = \dim_{\mathbbm{C}} M$, $p$ is the Fano index of $(M,g,J)$, $q \in \mathbbm{Z}$ is such that ${\bf{L}}^{q}$ is the auxiliary line bundle that induces a ${\rm{Spin}}^c$ structure, where ${\bf{L}}^{p} = {\bf{K}}_{M}$, satisfying $|q| \leq p$ and $p + q \in 2\mathbbm{Z}$. Further, the Kähler-Einstein metric $g$ has scalar curvature equal to $4m(m+1)$. It is worth pointing out that these restrictive conditions on $(M,g,J)$ and ${\bf{L}}$ are assumed in order to ensure the existence of Kählerian Killing ${\rm{Spin}}^{c}$ spinors, see for instance \cite{hijazi2006spin}. The existence of such spinors is a fundamental ingredient to obtain the lower bound described in Eq. (\ref{eqlb}).

A natural question which is motivated by the Nakad and Pilca construction is the following:

\begin{problem}
\label{problem1}
To what extent can we derive suitable lower bounds for the spectrum of the ${\rm{Spin}}^{c}$ Dirac operator $\mathcal{D}_{A}$ on compact Kähler manifolds if we drop the Kähler-Einstein assumption and allow for a general auxiliary line bundle?
\end{problem}
In a general setting, from an analytical perspective, the above problem is highly non-trivial. When we drop the Kähler-Einstein condition and use a general line bundle, the classic Schrödinger-Lichnerowicz formula
\begin{equation}
    {\mathcal{D}_{A}}^2 = {\nabla^A}^* \nabla^{A}+ \frac{1}{4}S_{g}1_{\mathcal{S}(M)} + \frac{1}{2} F_A,
\end{equation}
gives rise to cross-terms involving the non-constant scalar curvature and the curvature of the connection $A$ on the auxiliary line bundle ${\bf{L}}$. 

In an attempt to provide a satisfactory answer for the above question, in this paper, we restrict our attention to the class of compact Kähler manifolds provided by rational homogeneous varieties (complex flag varieties). In this context, we have a precise and constructive description for the underlying Picard group in terms of representation theory of semisimple Lie groups and Lie algebras, e.g. \cite{AZAD} and \cite{correa2023deformed}. These tools, combined with the invariance, allows us to establish the following results:
\begin{enumerate}
\item[(i)] Given a rational homogeneous variety $(X,\omega)$, such that $\omega$ is any invariant Kähler form, we provide a complete description for the spectrum of every closed invariant real $(1,1)$-form viewed as an operator acting on spinor bundles;
\item[(ii)] We provide a complete description for the spectrum of the Weitzenböck remainder operator 
\begin{center}
${\mathcal{D}}_{A}^{2} - {\bf{\Delta}}_{A} \colon \Gamma^{\infty}(\mathcal{S}(X)) \to \Gamma^{\infty}(\mathcal{S}(X))$, 
\end{center}
in terms of the underlying Lie-theoretic data of $(X,\omega)$, including an explicit formula for its smallest eigenvalue.
\item[(iii)] Combining the Atiyah-Singer index theorem with the Borel-Weil-Bott theorem, we provide a complete classification of ${\rm{Spin}}^{c}$ structures on rational homogeneous varieties which admit harmonic spinors. In this last setting, we present an explicit formula for the index of the associated ${\rm{Spin}}^{c}$ Dirac operator in terms of Lie theory.
\item[(iv)] As a consequence of (i) and (ii), we derive a general formula for a lower bound of ${\rm{Spec}}(\mathcal{D}_{A}^{2})$ which depends on the choice of $\omega$ and the auxiliary line bundle ${\bf{L}}$.
\end{enumerate}

In order to state our main results, let us introduce some concepts. A rational homogeneous variety can be described as a quotient $X_{P} = G^{\mathbbm{C}}/P$, where $G^{\mathbbm{C}}$ is a semisimple complex algebraic group with Lie algebra $\mathfrak{g}^{\mathbbm{C}} = {\rm{Lie}}(G^{\mathbbm{C}})$, and $P$ is a parabolic Lie subgroup (Borel-Remmert \cite{BorelRemmert}). Regarding $G^{\mathbbm{C}}$ as a complex analytic space, without loss of generality, we may assume that $G^{\mathbbm{C}}$ is a connected, simply connected, complex simple Lie group. Fixing a compact real form $G \subset G^{\mathbbm{C}}$, one can consider $X_{P} = G/G \cap P$ as a homogeneous $G$-space. In this work, we are interested in $G$-invariant geometric structures on $X_{P}$.

By choosing a Cartan subalgebra $\mathfrak{h} \subset \mathfrak{g}^{\mathbbm{C}}$ and a simple root system $\Delta \subset \mathfrak{h}^{\ast}$, up to conjugation, we have that $P \subset G^{\mathbbm{C}}$ is completely determined by some $I \subset \Delta$, e.g. \cite[\S 3.1]{Akhiezer}. In this setting, considering the associated fundamental weights $\varpi_{\alpha} \in \mathfrak{h}^{\ast}$, $\alpha \in \Delta$, it can be shown that  
\begin{equation}
{\rm{Pic}}(X_{P}) \cong H^{1,1}(X_{P},\mathbbm{Z}) \cong \Lambda_{P}:= \bigoplus_{\alpha \in \Delta \backslash I}\mathbbm{Z}\varpi_{\alpha},
\end{equation}
see for instance  \cite{correa2023deformed}. The aforementioned isomorphisms give us the assignments 
\begin{equation}
\label{map11character}
\begin{cases}

[\omega] \in H^{1,1}(X_{P},\mathbbm{Z}) \mapsto \phi([\omega]) \in \Lambda_{P},\\

{\bf{L}} \in {\rm{Pic}}(X_{P}) \mapsto \phi({\bf{L}}) \in \Lambda_{P}.
\end{cases}
\end{equation}
On a rational homogeneous variety $X_{P}$ we have a canonical ${\rm{Spin}}^{c}$ structure for which the associated complex spinor bundle is defined by 
\begin{equation}
\mathcal{S}(X_{P})_{\text{can.}} \cong \textstyle{\bigwedge^{0,\ast}T^{\ast}X_{P}},
\end{equation}
e.g. \cite{friedrich2000dirac}. Furthermore, it can be shown that every complex spinor bundle over $X_{P}$ associated to some ${\rm{Spin}}^{c}$ structure is obtained, up to isomorphism, from the canonical one together with a holomorphic line bundle ${\bf{L}} \in {\rm{Pic}}(X_{P})$, satisfying the following condition
\begin{equation}
c_{1}({\bf{L}}) = w_{2}(X_{P}) \ (\ {\rm{mod}} \ 2).
\end{equation}
We shall denote by ${\rm{Spin}}^{c}(X_{P}) \subset {\rm{Pic}}(X_{P})$ the set of isomorphism classes of line bundles satisfying the above condition, in other words, ${\rm{Spin}}^{c}(X_{P})$ parametrizes the isomorphism classes of ${\rm{Spin}}^{c}$ structures on $X_{P}$.

In the above setting, from ${\bf{L}} \in {\rm{Spin}}^{c}(X_{P})$, we obtain a complex spinor bundle by twisting the canonical complex spinor bundle with the line bundle ${\bf{E}} = \sqrt{{\bf{L}} \otimes {\bf{K}}_{X_{P}}}$, i.e., the complex spinor bundle associated with the ${\rm{Spin}}^{c}$ structure defined by ${\bf{L}}$ is given by
\begin{equation}
\mathcal{S}(X_{P}) = \mathcal{S}(X_{P})_{\text{can.}} \otimes {\bf{E}}.
\end{equation}
Denoting by $\Phi = \Phi^{+} \cup \Phi^{-}$ the root system associated with $\Delta$, considering  the set of roots $\Phi_{I}^{\pm}:= \Phi^{\pm} \backslash \langle I \rangle^{\pm}$, and setting 
\begin{equation}
\delta_{P} = \sum_{\alpha \in \Phi_{I}^{+}} \alpha \ \ \ \ \ \ \ \delta^{+} = \frac{1}{2} \sum_{\alpha \in \Phi^{+}}\alpha,
\end{equation}
we have the following results.
\begin{thm}
\label{theoremA}
Let $(X_{P},\omega)$ be a rational homogeneous variety, such that $\omega$ is a $G$-invariant K\"{a}hler metric, and let ${\bf{L}} \in {\rm{Spin}}^{c}(X_{P})$ be a ${\rm{Spin}}^{c}$ structure with associated spinor bundle $\mathcal{S}(X_{P})$. Then, the spectrum of the operator $\theta \colon \Gamma^{\infty}(\mathcal{S}(X_{P})) \to \Gamma^{\infty}(\mathcal{S}(X_{P}))$, induced by a closed $G$-invariant real $(1,1)$-form $\theta \in \Omega^{1,1}(X_{P})$, is given explicitly by
\begin{equation}
{\rm{Spec}}(\theta) = \Bigg \{ \sqrt{-1}\sum_{\beta \in \Phi_{I}^{+}} \epsilon_{\beta}\frac{\langle \phi([\theta]),\beta^{\vee} \rangle}{\langle \phi([\omega]),\beta^{\vee} \rangle} \ \ \Bigg | \ \ \epsilon_{\beta} = \pm 1, \forall \beta \in \Phi_{I}^{+} \ \Bigg \},
\end{equation}
such that $\phi([\theta]), \phi([\omega]) \in \Lambda_{P} \otimes \mathbbm{R}$.
\end{thm}

From Theorem \ref{theoremA} and the previous ideas, we have the following result that encompasses several important global aspects of ${\rm{Spin}}^{c}$ geometry on rational homogeneous varieties in terms of the combinatorics of roots and weights.

\begin{thm}
\label{theoremB}
Let $(X_{P},\omega)$ be a rational homogeneous variety, such that $\omega$ is a $G$-invariant K\"{a}hler metric. Then, the following hold:
\begin{enumerate}
\item[(A)] ${\bf{L}} \in {\rm{Spin}}^{c}(X_{P})$ if, and only if,
\begin{equation}
\int_{\mathbbm{P}_{\alpha}^{1}}c_{1}({\bf{L}}) = \langle \delta_{P},\alpha^{\vee} \rangle \ ({\rm{mod}} \ 2), \ \ \forall \alpha \in \Delta \backslash I,
\end{equation}
where $[\mathbbm{P}_{\alpha}^{1}] \in \pi_{2}(X_{P})$, $\alpha \in \Delta \backslash I$, are the generators of the cone of curves ${\rm{NE}}(X_{P})$;
\item[(B)] If ${\mathcal{D}}_{A} \colon \Gamma^{\infty}({\mathcal{S}}(X_{P})) \to \Gamma^{\infty}({\mathcal{S}}(X_{P}))$ is the Dirac operator defined by the underlying Yang-Mills connection $A$ on some ${\bf{L}} \in {\rm{Spin}}^{c}(X_{P})$, then
\begin{equation}
{\rm{Spec}}\big ( {\mathcal{D}}_{A}^{2} - {\bf{\Delta}}_{A}\big ) = \Bigg \{ \pi  \sum_{\beta \in \Phi_{I}^{+}}\frac{\langle \delta_{P},\beta^{\vee} \rangle + \epsilon_{\beta} \langle \phi({\bf{L}}), \beta^{\vee}\rangle}{\langle \phi([\omega]), \beta^{\vee}\rangle} \ \ \Bigg | \ \ \epsilon_{\beta} = \pm 1, \forall \beta \in \Phi_{I}^{+}\Bigg \},
\end{equation}
such that $\phi({\bf{L}}) \in \Lambda_{P}$ and $\phi([\omega]) \in \Lambda_{P} \otimes \mathbbm{R}$;
\item[(C)] In the setting of item (B), we have the smallest eigenvalue of the operator defined by the Weitzenböck remainder ${\mathcal{D}}_{A}^{2} - {\bf{\Delta}}_{A} \colon \Gamma^{\infty}({\mathcal{S}}(X_{P})) \to \Gamma^{\infty}({\mathcal{S}}(X_{P}))$ given by
\begin{equation}
\label{smallestWR}
\lambda_{\rm{min}} \big ({\mathcal{D}}_{A}^{2} - {\bf{\Delta}}_{A} \big ) =  \pi  \sum_{\beta \in \Phi_{I}^{+}}\frac{\langle \delta_{P},\beta^{\vee} \rangle -  |\langle \phi({\bf{L}}), \beta^{\vee}\rangle|}{\langle \phi([\omega]), \beta^{\vee}\rangle}.
\end{equation}

\item[(D)] In the setting of item (B), the Dirac operator ${\mathcal{D}}_{A} \colon \Gamma^{\infty}({\mathcal{S}}(X_{P})) \to \Gamma^{\infty}({\mathcal{S}}(X_{P}))$ admits a harmonic spinor if and only if the weight
\begin{equation}
\phi({\bf{E}}) = (1/2)(\phi({\bf{L}}) - \delta_{P}) \in \Lambda_{P},
\end{equation}
is a regular weight, i.e., if $\exists ! w \in \mathscr{W}_{\mathfrak{g}^{\mathbbm{C}}}$, such that $w \star \phi({\bf{E}}) \in \Lambda^{+}$. In particular, if $\phi({\bf{E}})$ is a regular weight, then 
\begin{equation}
{\rm{Index}}(\mathcal{D}_{A}) =  (-1)^{\ell(w)}\frac{\displaystyle{\Pi_{\alpha \in \Phi^{+}}\langle w(\phi({\bf{E}}) + \delta^{+}), \alpha \rangle}}{\displaystyle{\Pi_{\alpha \in \Phi^{+}}\langle \delta^{+},\alpha \rangle}},
\end{equation}
where $\ell(w)$ is the length of the unique element $w \in \mathscr{W}_{\mathfrak{g}^{\mathbbm{C}}}$, such that $w \star \phi({\bf{E}}) \in \Lambda^{+}$.
\end{enumerate}
\end{thm}

In the above theorem, from item (B), we have a complete description for the spectrum of the Weitzenböck remainder operator ${\mathcal{D}}_{A}^{2} - {\bf{\Delta}}_{A}$. Also, in item (C), we provide an explicit formula for the smallest eigenvalue of this operator in terms of the Lie-theoretic data underlying $X_{P}$. 

The dependence on $\mathbf{L}$ and $[\omega]$ in the formula for the smallest eigenvalue (Eq. \ref{smallestWR}) is particularly significant: it allows one to determine, through a straightforward computation, whether the harmonic spinor equation
\begin{equation}
\mathcal{D}_{A} \psi = 0, \ \ \psi \in \Gamma^{\infty}(\mathcal{S}(X_{P})),
\end{equation}
can not be solved for an arbitrary ${\rm{Spin}}^{c}$ structure ${\bf{L}}$ on $(X_{P},\omega)$. In fact, it follows from Eq. (\ref{smallestWR}) and from \cite[Theorem 1.1]{hitchin1974harmonic} that, if 
\begin{equation}
 \pi  \sum_{\beta \in \Phi_{I}^{+}}\frac{\langle \delta_{P},\beta^{\vee} \rangle -  |\langle \phi({\bf{L}}), \beta^{\vee}\rangle|}{\langle \phi([\omega]), \beta^{\vee}\rangle} > 0,
 \end{equation}
then the ${\rm{Spin}}^{c}$ structure defined by ${\bf{L}}$ on $(X_{P},\omega)$ admits no harmonic spinors. 

The result of item (D) of Theorem \ref{theoremB} completely classifies the ${\rm{Spin}}^{c}$ structures ${\bf{L}}$ on $X_P$ that admit harmonic spinors, in summary:
\begin{equation}
\exists \psi \in \Gamma^{\infty}(\mathcal{S}(X_{P})), \ \ \mathcal{D}_{A}\psi = 0 \iff \ \ \exists ! w \in \mathscr{W}_{\mathfrak{g}^{\mathbbm{C}}}, \text{ s.t. } w \star \phi({\bf{E}}) \in \Lambda^{+}.
\end{equation}
As we see from above, we have a concrete and checkable criterion using the combinatorics of roots and the Weyl group for the existence of harmonic spinors. Here we consider the shifted action
\begin{equation}
w \star \lambda := w(\lambda + \delta^{+}) - \delta^{+},
\end{equation}
for every $\lambda \in \mathfrak{h}^{\ast}$ and every $w \in \mathscr{W}_{\mathfrak{g}^{\mathbbm{C}}}$. The explicit index formula provided by item (D) is achieved by means of two classical theorems: the Atiyah-Singer index theorem (\cite{AtiyahSinger1968c}) and the Borel-Weil-Bott theorem \cite{serre1954representations}, \cite{bott1957homogeneous}, \cite{demazure1968demonstration}, \cite{demazure1976very}. 

It is worth mentioning that the key ideas to prove item (D) fit into a rich tradition of results linking Dirac operators on homogeneous spaces with representation theory via the Borel–Weil–Bott theorem. In the setting of homogeneous spaces $G/H$, Slebarski \cite{Slebarski1987,Slebarski1987b} and later Landweber \cite{Landweber2000} studied the kernels of twisted Dirac operators, showing that they can be described in terms of representation-theoretic data. More recently, Huang and Pand{\v{z}}i{\'c} \cite{HuangPandzic2006} developed the powerful framework of Dirac cohomology, which provides a conceptual generalization of the Borel–Weil–Bott theorem. Hong \cite{Hong2014} gave a proof of the Borel-Weil-Bott theorem using Kostant's cubic Dirac operator. Our contribution extends these results in three directions: (i) we work on arbitrary rational homogeneous varieties $X_{P}$ (not necessarily full flag manifolds\footnote{Full flag manifolds always admit a (real) ${\rm{Spin}}$ structure, see \cite[p. 673]{alekseevsky2019spin})}), (ii) we consider ${\rm{Spin}}^{c}$ structures, and (iii) we provide an explicit closed formula for the index in terms of the length of a Weyl group element and root products, which, to the best of our knowledge, is new in this generality.

Our next result can be understood as an application of the machinery developed in Theorem \ref{theoremB} to tackle Problem \ref{problem1} in the setting of rational homogeneous varieties. More precisely, we have the following.

\begin{thm}
\label{theoremC}
Let $(X_{P},\omega)$ be a rational homogeneous variety, such that $\omega$ is a $G$-invariant K\"{a}hler metric, and let ${\bf{L}} \in {\rm{Spin}}^{c}(X_{P})$ be a ${\rm{Spin}}^{c}$ structure with associated complex spinor bundle $\mathcal{S}(X_{P})$. If ${\mathcal{D}}_{A} \colon \Gamma^{\infty}({\mathcal{S}}(X_{P})) \to \Gamma^{\infty}({\mathcal{S}}(X_{P}))$ is the Dirac operator defined by the underlying Yang-Mills connection $A$ on ${\bf{L}}$, then any eigenvalue $\lambda$ of ${\mathcal{D}}_{A}$ satisfies the following
\begin{equation}
\label{lbcombinatoric}
\lambda^{2} \geq \pi  \sum_{\beta \in \Phi_{I}^{+}}\frac{\langle \delta_{P},\beta^{\vee} \rangle -  |\langle \phi({\bf{L}}), \beta^{\vee}\rangle|}{\langle \phi([\omega]), \beta^{\vee}\rangle},
\end{equation}
such that $\phi({\bf{L}}) \in \Lambda_{P}$ and $\phi([\omega]) \in \Lambda_{P} \otimes \mathbbm{R}$.
\end{thm}

The result above provides a general formula of a lower bound for the eigenvalues of the ${\rm{Spin}}^{c}$ Dirac operator $\mathcal{D}_{A}$. This lower bound generalizes the classical Friedrich inequality \cite{friedrich1980erste} to the $\mathrm{Spin}^{c}$ setting on rational homogeneous varieties, and it reduces to a purely combinatorial expression that depends explicitly on the choice of ${\bf{L}}$ and $[\omega]$. Notice that the lower bound is exactly the smallest eigenvalue $\lambda_{\rm{min}} \big ({\mathcal{D}}_{A}^{2} - {\bf{\Delta}}_{A} \big )$ of the associated Weitzenböck remainder provided in item (C) of Theorem \ref{theoremB}.

Although Theorem \ref{theoremC} provides a general method to obtain an explicit lower bound for the eigenvalues of ${\rm{Spin}}^{c}$ Dirac operators, it is worth observing the following. Let $p$ be the Fano index of $X_{P}$ and consider the primitive root ${\bf{L}} \in {\rm{Pic}}(X_{P})$ of ${\bf{K}}_{X_{P}}$, that is, ${\bf{L}}^{p} = {\bf{K}}_{X_{P}}$. Given $q \in \mathbbm{Z}$, such that $p + q \in 2\mathbbm{Z}$, we have ${\bf{L}}^{q} \in {\rm{Spin}}^{c}(X_{P})$. Further, taking the Kähler-Einstein metric $\omega$ with scalar curvature equal to $4m(m+1)$, it can be shown that 
\begin{equation}
\label{lbDirac}
\phi([\omega]) =\frac{\pi}{(m+1)}\delta_{P} \ \ \ \ \text{and} \ \ \ \ \phi({\bf{L}}) = -\frac{q}{p}\delta_{P},
\end{equation}
here we consider $m = \dim_{\mathbbm{C}}(X_{P})$. Applying the above data to Eq. (\ref{lbDirac}), we obtain 
\begin{equation}
    \lambda^{2} \geq \left(1 - \frac{|q|}{p}\right)m(m+1).
\end{equation}
If we suppose additionally that $|q| < p$, we are exactly under the hypotheses considered in the main result of \cite{nakad2015eigenvalue}. As it can be seen, the lower bound obtained above is slightly weaker than \eqref{eqlb}, in other words, the lower bound in Eq. (\ref{lbcombinatoric}) is not sharp.

\subsubsection*{Acknowledgments.} E. M. Correa is supported by S\~{a}o Paulo Research Foundation FAPESP grant 25/18843-1.

\section{Spin$^c$-structures and Dirac Operators}
In this section, we introduce the basic concepts of ${\rm{Spin}}^{c}$ structures required to prove our main results. For more details concerning the topics discussed in this section, we refer the reader to \cite{lawson2016spin}, \cite{friedrich2000dirac}, \cite{bourguignon2015spinorial}. 
\subsection{Generalities on Spin$^c$-structures} A smooth Riemannian manifold $(M,g)$ admits a (real) spin structure if it is orientable and its second Stiefel-Whitney class vanishes, i.e., $w_2(TM) = 0$. In that case, the set of isomorphism classes of spin structures is in one-to-one correspondence with $H^1(M; \mathbbm{Z}_2)$. Thus, the obstruction to the existence of a spin structure is purely topological \cite{lawson2016spin}.

Before we define what a ${\rm{Spin}}^{c}$ structure is, we need to define the ${\rm{Spin}}^{c}(n)$ group. Let ${\rm{Spin}}(n)$  be the spin group and ${\rm{U}}(1)$ the group of unitary complex numbers. Since ${\rm{Spin}}(n) \cap {\rm{U}}(1) = \lbrace 1, - 1 \rbrace$ we can define the Spin$^{c}(n)$ group as follows 
\begin{equation*}
    {\rm{Spin}}^{c}(n)=({\rm{Spin}}(n) \times {\rm{U}}(1)) / \lbrace \pm 1 \rbrace.
\end{equation*}
From this, we say that $M$ admits a ${\rm{Spin}}^{c}$ structure if there exists a principal $\text{Spin}^{c}(n)$-bundle $P_{{\rm{Spin}}^{c}(n)}M$ and a bundle map
$
\Lambda : P_{{\rm{Spin}}^{c}(n)}M \longrightarrow P_{{\rm{SO}}(n)}M \times P_{\text{U}(1)}M,
$
that is equivariant with respect to the natural homomorphism $\lambda^c: {\rm{Spin}}^{c}(n) \to {\rm{SO}}(n) \times \text{U}(1)$. More precisely, the following diagram commutes for every $\tilde{u} \in P_{{\rm{Spin}}^{c}(n)}M$ and every $a \in {\rm{Spin}}^{c}(n)$:
\begin{center}
    \begin{tikzcd}[scale=0.7, font=\scriptsize, row sep=scriptsize, column sep=scriptsize]
{\rm{Spin}}^{c}(n) \arrow[d, "a \mapsto \Tilde{u}a" swap] \arrow[rr,"\lambda^{c}"]  &  & {\rm{SO}}(n) \times {\rm{U}}(1)\arrow[d, "{(A,z) \mapsto \Lambda(\Tilde{u})(A,z)}"]  
   \\ 
P_{{\rm{Spin}}^{c}(n)}M   \arrow[rd] \arrow[rr, "\Lambda"]       & & \arrow[ld] P_{{\rm{SO}}(n)}M \times P_{{\rm{U}}(1)}M    \\
 & M  & \
\end{tikzcd}
\end{center}
Here we observe that $P_{\text{U}(1)}M$ is the principal $\text{U}(1)$-bundle associated to an auxiliary complex line bundle ${\bf{L}}$, i.e.
\begin{equation}
{\bf{L}} = P_{\text{U}(1)}M \times_{{\rm{U}}(1)}\mathbbm{C},
\end{equation}
and $P_{{\rm{SO}}(n)}M \times P_{\text{U}(1)}M$ denotes the Whitney product bundle \cite{bourguignon2015spinorial}. 

Each ${\rm{Spin}}^{c}$ structure on a Riemannian manifold $M$ induces a complex spinor bundle
\begin{equation}
\mathcal{S}(M)=P_{{\rm{Spin}}^{c}(n)}M\times_{\rho}\Delta_n,
\end{equation}
where $\Delta_n$ is a complex $\mathbbm{C}l_n$-module and $\rho$ denotes the spin representation of
${\rm{Spin}}^c(n)$. The structure of $\Delta_n$ is well known (see \cite{friedrich2000dirac}):
if $n=2m$, then $\Delta_n\simeq\mathbbm{C}^{2^m}$, while if $n=2m+1$, one has
$\Delta_n\simeq\mathbbm{C}^{2^m}\oplus\mathbbm{C}^{2^m}$.
In what follows, we shall focus on the even-dimensional case $n=2m$, for which, at each
point $x\in M$, the spinor space admits the tensor product decomposition
\begin{equation}
\mathcal{S}(M)_x\simeq\underbrace{\mathbbm{C}^2\otimes\cdots\otimes\mathbbm{C}^2}_{m\text{ times}}.
\end{equation}

This vector bundle admits a Hermitian inner product, which we denote by $( \cdot, \cdot )$. In order to introduce a connection in $\mathcal{S}(M)$ that respects the Clifford product, it is necessary to define the following bundle,
\begin{equation}
\mathbbm{C}l(M) = P_{{\rm{SO}}(n)} M \times_{\mu} \mathbbm{C}l_n,
\end{equation}
here $\mu$ denotes the natural action of ${\rm{SO}}(n)$ on $\mathbbm{C}l_{n}$, and $\mathcal{S}(M)$ becomes a $\mathbbm{C}l(M)$-module. This bundle is known as the Clifford bundle.

Using characteristic classes, we can formulate the following result \cite{lawson2016spin, friedrich2000dirac} which characterizes the existence of ${\rm{Spin}}^{c}$ structures.
\begin{proposition}
\label{prop2}
    Let $(M,g)$ be an oriented Riemannian manifold. Then we obtain the following equivalence 
    \begin{enumerate}
        \item $M$ admits a $\rm{Spin}^{c}$-structure.
        \item There exists a line bundle ${\bf{L}}$ over $M$ such that $w_2(TM) \equiv c_{1}({\bf{L}}) \pmod{2}$.
    \end{enumerate}
    where $w_{2}(TM)$ is the second Stiefel-Whitney class of the tangent bundle, and $c_{1}({\bf{L}}) \in H^2(M; \mathbbm{Z})$ denotes the first Chern class of ${\bf{L}}$.
\end{proposition}
\begin{remark}
Every ${\rm{Spin}}$ manifold is automatically a ${\rm{Spin}}^{c}$ manifold, but the converse is not true. For example, $\mathbbm{C}P^n$ admits a ${\rm{Spin}}^{c}$ structure for all $n$, but it admits a ${\rm{Spin}}$ structure only when $n$ is odd, see for instance \cite{hitchin1974harmonic}.
\end{remark}

In general, every Hermitian manifold $(M^{2m}, g, J)$ admits a canonical ${\rm{Spin}}^{c}$ structure. In fact, by choosing the auxiliary line bundle ${\bf{L}} = {\bf{K}}_{M}^{-1}$ (the anticanonical bundle), since 
\begin{center}
$w_2(TM) \equiv c_{1}(M) \pmod{2}$, 
\end{center}
and $c_{1}(M) = c_{1}({\bf{K}}_{M}^{-1})$, it follows from Proposition  \ref{prop2} that there exists a $\rm{Spin}^{c}$-structure on $(M^{2m}, g, J)$ with auxiliary line bundle $ {\bf{K}}_{M}^{-1}$. In this case, we have the following description for the associated complex spinor bundle
\begin{equation}
\mathcal{S}(M)_{\text{can.}} \cong \textstyle{\bigwedge^{0,\ast}T^{\ast}M}.
\end{equation}
From above we can show that every ${\rm{Spin}}^{c}$ structure on $(M^{2m}, g, J)$ is obtained by twisting the above one by the line bundle ${\bf{E}} \in {\rm{Pic}}(M)$ satisfying 
\begin{equation}
{\bf{E}}^{2}:= {\bf{E}} \otimes {\bf{E}} = {\bf{L}} \otimes {\bf{K}}_{M}.
\end{equation}
In the above case, the associated complex spinor bundle satisfies the following 
\begin{equation}
\mathcal{S}(M) \cong \mathcal{S}(M)_{\text{can.}} \otimes {\bf{E}},
\end{equation}
in other words, a ${\rm{Spin}}^{c}$ structure on $(M^{2m}, g, J)$ is completely determined by its auxiliary line bundle ${\bf{L}}$, for more details, see for instance \cite{friedrich2000dirac}. 

Given a Hermitian manifold $(M^{2m}, g, J)$, such that $H^{2}(M,\mathbbm{Z})$ has no 2-torsion, we shall denote by
\begin{equation}
{\rm{Spin}}^{c}(M): = \big \{ {\bf{L}} \in {\rm{Pic}}(M) \ | \ w_2(TM) \equiv c_{1}({\bf{L}}) \pmod{2} \big \},
\end{equation}
the set of isomorphism classes of ${\rm{Spin}}^{c}$ structures on $(M^{2m}, g, J)$.
\subsection{Spin$^{c}$ Dirac operator and Weitzenböck remainder}

Let $(M,g)$ be an oriented Riemannian manifold equipped with a ${\rm{Spin}}^{c}$ structure. In this setting, let us denote by ${\bf{L}}$ the associated auxiliary line bundle.

A natural connection $\nabla^{A}$ on $\mathcal{S}(M)$ is induced by the Levi-Civita connection $\nabla^M$ on $M$ and  a connection $\nabla = d + A$ on the auxiliary line bundle ${\bf{L}}$. This connection satisfies the Leibniz rule with respect to the Clifford action:
\begin{equation}
\label{compc}
    \nabla^{A}_Y(\sigma \cdot \psi) = (\nabla^{A}_Y \sigma) \cdot \psi + \sigma \cdot (\nabla^{A}_Y \psi),
\end{equation}
for all $Y \in \Gamma(TM)$, $\sigma \in \Gamma(\mathbbm{C}l(M))$, and $\psi \in \Gamma(\mathcal{S}(M))$. Moreover, $\nabla^{A}$ is compatible with the Hermitian inner product on $\mathcal{S}(M)$.

There is a local description of this connection. Consider a local section $\tilde{u} \in \Gamma(P_{\operatorname{Spin}^c(n)}(M))$ such that $\Lambda(\tilde{u}) = (u, \gamma)$, where $u = (e_1, \dots, e_n)$ is a local orthonormal frame of $TM$ and $\gamma \in \Gamma(L)$ is a local section. A local spinor field can be expressed as $\Psi = [\tilde{u}, \psi]$, where $\psi$ is a function with values in $\Delta_n$. Then the connection acts locally as:
\begin{equation}
\nabla^{A}_Y \Psi = \left[ \tilde{u},\ Y(\psi) + \frac{1}{2} \sum_{i < j} g(\nabla_Y e_i, e_j) \, e_i \cdot e_j \cdot \psi + \sqrt{-1} \, A(Y) \cdot \psi \right],
\end{equation}
$\forall Y \in \Gamma(TM)$, $\cdot$ denotes Clifford multiplication, and $A$ is the local connection 1-form. 

From this, we can define  the Spin$^c$ Dirac operator, $\mathcal{D}^{A}: \Gamma^{\infty}(\mathcal{S}(M)) \rightarrow \Gamma^{\infty}(\mathcal{S}(M))$ by 
\begin{equation}
\mathcal{D}_{A}\psi = \sum_{i=1}^{n}e_{i}\cdot \nabla^{A}_{e_{i}}\psi. 
\end{equation}
In the above setting, we have the following result. 

\begin{proposition}
\label{l2product}
    $\mathcal{D}_{A}$ is a self-adjoint operator with respect to the scalar product
    \begin{equation}
    \langle \psi,\varphi \rangle:= \int_{M}( \psi, \varphi)d\mathrm{vol_{g}},
    \end{equation}  
where $\psi,\varphi \in \Gamma^{\infty}(\mathcal{S}(M))$ are compactly supported sections, i.e., $\langle \mathcal{D}_{A}\psi,\varphi\rangle= \langle \psi,\mathcal{D}_{A}\varphi \rangle$.
\end{proposition}

The next result provides a fundamental relation between the ${\rm{Spin}}^{c}$ Dirac operator and the geometry of the manifold.

\begin{theorem}[Schr\"{o}dinger-Lichnerowicz formula]
By keeping the previous notation, it follows that
\begin{equation}
\label{maineq}
    {\mathcal{D}_{A}}^{2} = {\nabla^A}^* \nabla^{A}+ \frac{1}{4}S_{g}1_{\mathcal{S}(M)} + \frac{1}{2} F_{A},
\end{equation}
where $S_{g}$ is the scalar curvature of $(M,g)$ and $F_{A} = dA$.
\end{theorem}
In the above theorem we consider the curvature $F_A$ of $A$ acting via Clifford multiplication extended to differential forms \cite{friedrich2000dirac}. 

In this paper, we also consider the following operator.

\begin{definition}
The Weitzenböck remainder of a ${\rm{Spin}}^{c}$ Dirac operator $\mathcal{D}^{A}$ is defined by the operator
\begin{equation}
{\mathcal{D}_{A}}^{2} - {\bf{\Delta}}_{A}  \colon \Gamma^{\infty}(\mathcal{S}(M)) \rightarrow \Gamma^{\infty}(\mathcal{S}(M)).
\end{equation}
such that ${\bf{\Delta}}_{A}  = {\nabla^A}^* \nabla^{A}$ is the Bochner-Laplacian on $\mathcal{S}(M)$.
\end{definition}

As we see from the Schr\"{o}dinger-Lichnerowicz formula, in order to study the spectrum of the ${\mathcal{D}_{A}}^{2} - {\bf{\Delta}}_{A} $, we need to deal with the operator $F_{A} \colon \Gamma^{\infty}(\mathcal{S}(M)) \rightarrow \Gamma^{\infty}(\mathcal{S}(M))$. We shall explore this action in the particular case that $(M^{2m}, g, J)$ is Kähler manifold. In this particular setting, choosing $A$ as Chern connection, it follows that $\theta = \frac{\sqrt{-1}F_{A}}{2\pi}$ is a $(1,1)$-form on $M$. For this particular class of $2$-forms on a Kähler manifold $(M^{2m}, g, J)$, we have the following standard result.

\begin{proposition}
\label{eigenvalue}
    Let $(M^{2m}, g, J)$ be a Kähler manifold of complex dimension $m$, with Kähler form $\omega = g(J \otimes {\rm{id}})$, and let $\theta$ be a real $(1,1)$-form on $M$. For any point $p \in M$, there exists a local holomorphic coordinate system $\{w_1, \dots, w_m\}$ centered at $p$ (i.e., $w_j(p) = 0$), such that, at the point $p$, we have
    \begin{equation}
    \label{eigenv}
      \theta_p = \frac{\sqrt{-1}}{2} \sum_{j=1}^n \lambda_j(p) \, dw_j \wedge d\bar{w}_j|_{p} \ \ \ \text{and} \ \ \ \omega_{p} = \frac{\sqrt{-1}}{2} \sum_{j=1}^n \, dw_j \wedge d\bar{w}_{j}|_{p},
    \end{equation}
where $\lambda_{1}(p), \dots, \lambda_{n}(p) \in \mathbbm{R}$ are the eigenvalues of the endomorphism $\omega^{-1} \circ \theta$ at $p \in M$.
\end{proposition}

Let $(M,J,g)$ be a Kähler manifold of complex dimension $m$. In the above setting, we can choose a local orthonormal frame $\{e_1,J(e_{1}), \ldots, e_{m}, J(e_{m})\}$ at $p \in M$, such that 
\begin{equation}
\theta =  \sum_{j=1}^{m} \lambda_{j}(p)e_{j} \wedge J(e_{j}) \ \ \ \text{and} \ \ \ \omega = \sum_{j=1}^{m} e_{j} \wedge J(e_{j}).
\end{equation}
Now consider a ${\rm{Spin}}^{c}$ structure on $(M^{2m}, J, g)$, with associated spinor bundle $\mathcal{S}(M)$, since the above frame is orthonormal, we may identify $e_{j} \wedge J(e_{j})$ with the Clifford product $e_{j} \cdot J(e_{j})$, so that $\theta$ acts on $\mathcal{S}(M)_{p}$ by
\begin{equation}
\theta\cdot \psi= \sum_{j=1}^{m} \lambda_{j}(p)  e_{j}\cdot J(e_{j})  \cdot \psi.
\end{equation}

By considering the Spin$^{c}$ representation $\rho$, as mentioned in the previous section, we have the operators $\rho(e_{j}\cdot J(e_{j}))$, $1\le j\le m$, commute and are diagonalizable. Then the spinor space $\mathcal S(M)_{p}$ decomposes as a direct sum of common eigenspaces indexed by
\begin{equation}
\epsilon=(\epsilon_1,\dots,\epsilon_m)\in\{\pm1\}^m, 
\end{equation}
i.e, it has a basis of $2^{m}$ spinors $\psi(\epsilon_{1},\ldots,\epsilon_{m})$, such that
\begin{equation}
e_{j}\cdot J(e_{j})\,\psi(\epsilon_{1},\ldots,\epsilon_{m})
=
\epsilon_j\,\sqrt{-1}\,\psi(\epsilon_1,\dots,\epsilon_m),
\end{equation}
where $\epsilon_j\in\{\pm1\}$, for more details, see \cite[Page 7]{hitchin1974harmonic}. In particular, considering $\theta = \frac{\sqrt{-1}}{2\pi}F_{A}$, we conclude that 
\begin{equation}
{\mathcal{D}_{A}}^{2} - {\bf{\Delta}}_{A} = \frac{1}{4}S_{g}(p)1_{\mathcal{S}(M)} + \frac{\pi}{\sqrt{-1}}\sum_{j=1}^{m} \lambda_{j}(p)  e_{j}\cdot J(e_{j}),
\end{equation}
at $p \in M$. As we shall see, the above expression allows us to describe the spectrum of ${\mathcal{D}_{A}}^{2} - {\bf{\Delta}}_{A}$ in the homogeneous setting.

\subsection{Atiyah-Singer index theorem}

Let $(M,g)$ be an oriented Riemannian manifold equipped with a ${\rm{Spin}}^{c}$ structure with auxiliary line bundle ${\bf{L}}$. In this case, we have the following fact about the associated complex spinor bundle
\begin{equation}
\begin{cases} \mathcal{S}(M) \ \text{is irreducible if} \ \dim(M) \ \text{is odd},\\
\\
\mathcal{S}(M) = \mathcal{S}(M)^{+} \oplus \mathcal{S}(M)^{-} \ \text{if} \ \dim(M) \ \text{is even}.
\end{cases}
\end{equation}
In particular, if $\dim(M) = 2m$, then the associated ${\rm{Spin}}^{c}$ Dirac operator $\mathcal{D}_{A}$ has the following decomposition
\begin{equation}
\mathcal{D}_{A} = \begin{pmatrix} 0 & \mathcal{D}_{A}^{-} \\
\mathcal{D}_{A}^{+} & 0\end{pmatrix},
\end{equation}
such that $\mathcal{D}_{A}^{\pm} \colon \Gamma(\mathcal{S}(M)^{\pm}) \to \Gamma(\mathcal{S}(M)^{\mp})$. In the above setting, we have the following definition.
\begin{definition}
Let $(M,g)$ be a compact even-dimensional manifold which admits a ${\rm{Spin}}^{c}$ structure. We define the index of the associated ${\rm{Spin}}^{c}$ Dirac operator $\mathcal{D}_{A}$ by 
\begin{equation}
{\rm{Index}}(\mathcal{D}_{A}):= \dim_{\mathbbm{C}}(\ker(D_{A}^{+})) - \dim_{\mathbbm{C}}(\ker(D_{A}^{-})).
\end{equation}
\end{definition}

Now we consider the following important result.

\begin{theorem}[Atiyah-Singer, \cite{AtiyahSinger1968c}]
Let $(M,g)$ be a compact even-dimensional manifold that admits a ${\rm{Spin}}^{c}$ structure with auxiliary line bundle ${\bf{L}}$. Then, the index of the associated ${\rm{Spin}}^{c}$ Dirac operator $\mathcal{D}_{A}$ satisfies
\begin{equation}
\label{indexformula}
{\rm{Index}}({\mathcal{D}}_{A}) = \int_{M}{\rm{e}}^{\frac{c_{1}({\bf{L}})}{2}}{\widehat{A}}(M),
\end{equation}
where ${\widehat{A}}(M)$ is the $\widehat{A}$-class of $TM$ and $c_{1}(\bf{L})$ is the first Chern class of ${\bf{L}}$.
\end{theorem}

In the Kähler case, the kernel of the ${\rm{Spin}}^{c}$ Dirac operator admits a beautiful description in terms of the Dolbeault cohomology of the auxiliary line bundle ${\bf{E}}$.

Let $(M^{2m}, g, J)$ be a compact Kähler manifold and let a ${\rm{Spin}}^{c}$ structure with auxiliary line bundle is ${\bf{L}}$. As we have seen, the spinor bundle decomposes as  
\begin{equation}
\mathcal{S}(M) \cong \Lambda^{0,\ast}T^{\ast}M \otimes {\bf{E}},
\end{equation}  
such that ${\bf{E}}^{2} = {\bf{K}}_{M} \otimes {\bf{L}}$, and the ${\rm{Spin}}^{c}$ Dirac operator becomes  
\begin{equation}
\label{DiracDolbeault}
\mathcal{D}_{A} = \sqrt{2}\,\big(\bar{\partial}_{{\bf{E}}} + \bar\partial_{{\bf{E}}}^{\ast}\big),
\end{equation}  
see for instance \cite{duistermaat1996heat}, where $\bar{\partial}_{\bf{E}}$ is the Dolbeault operator coupled to ${\bf{E}}$, that is 
\begin{equation}
\bar{\partial}_{\bf{E}} \colon \Gamma(\Lambda^{0,\bullet}T^{\ast}M \otimes {\bf{E}}) \to \Gamma(\Lambda^{0,\bullet+1}T^{\ast}M \otimes {\bf{E}}).
\end{equation}
Now we consider the following definition.

\begin{definition}
We say that $\psi \in \Gamma^{\infty}(\mathcal{S}(M))$ is a harmonic spinor if $\mathcal{D}_{A}\psi = 0$.
\end{definition}

As a consequence of the previous ideas, the space of harmonic spinors is isomorphic to the total Dolbeault cohomology of ${\bf{E}}$:
\begin{equation}
\ker(\mathcal{D}_{A}) \;\cong\; \bigoplus_{p=0}^{m} H^p\bigl(M,{\bf{E}}\bigr).
\end{equation}  
The $\mathbb{Z}_{2}$-grading of the spinor bundle corresponds to the parity of the form degree:
\begin{equation}
\ker(\mathcal{D}_{A}^{+}) \cong \bigoplus_{p\ \text{even}} H^{p}(M,{\bf{E}}), \qquad 
\ker(\mathcal{D}_{A}^{-})\cong \bigoplus_{p\ \text{odd}} H^{p}(M,{\bf{E}}),
\end{equation}  
see for instance \cite{duistermaat1996heat}. Thus the existence of harmonic spinors is completely determined by the cohomology groups of the line bundle ${\bf{E}}$. Moreover, the index of the ${\rm{Spin}}^{c}$ Dirac operator coincides with the holomorphic Euler characteristic of ${\bf{E}}$:
\begin{equation}
{\rm{Index}}(\mathcal{D}_{A}) = \chi(M,{\bf{E}}) = \sum_{p=0}^{m} (-1)^{p} \dim_{\mathbbm{C}} H^{p}(M,{\bf{E}}).
\end{equation} 

\section{Rational homogeneous varieties}
\label{generalities}
In this section, we review some basic generalities about flag varieties. For more details on the subject presented in this section, we suggest \cite{Akhiezer}, \cite{Flagvarieties}, \cite{HumphreysLAG}, \cite{BorelRemmert}.
\subsection{The Picard group of flag varieties}
\label{subsec3.1}
Let $G^{\mathbbm{C}}$ be a connected, simply connected, and complex Lie group with simple Lie algebra $\mathfrak{g}^{\mathbbm{C}}$. By fixing a Cartan subalgebra $\mathfrak{h}$ and a simple root system $\Delta \subset \mathfrak{h}^{\ast}$, we have a triangular decomposition of $\mathfrak{g}^{\mathbbm{C}}$ given by
\begin{center}
$\mathfrak{g}^{\mathbbm{C}} = \mathfrak{n}^{-} \oplus \mathfrak{h} \oplus \mathfrak{n}^{+}$, 
\end{center}
where $\mathfrak{n}^{-} = \sum_{\alpha \in \Phi^{-}}\mathfrak{g}_{\alpha}$ and $\mathfrak{n}^{+} = \sum_{\alpha \in \Phi^{+}}\mathfrak{g}_{\alpha}$, here we denote by $\Phi = \Phi^{+} \cup \Phi^{-}$ the root system associated with the simple root system $\Delta \subset \mathfrak{h}^{\ast}$. Let us denote by $\kappa$ the Cartan-Killing form of $\mathfrak{g}^{\mathbbm{C}}$. From this, for every  $\alpha \in \Phi^{+}$, we have $h_{\alpha} \in \mathfrak{h}$, such  that $\alpha = \kappa(\cdot,h_{\alpha})$, and we can choose $x_{\alpha} \in \mathfrak{g}_{\alpha}$ and $y_{-\alpha} \in \mathfrak{g}_{-\alpha}$, such that $[x_{\alpha},y_{-\alpha}] = h_{\alpha}$. From these data, we can define a Borel subalgebra\footnote{A maximal solvable subalgebra of $\mathfrak{g}^{\mathbbm{C}}$.} by setting $\mathfrak{b} = \mathfrak{h} \oplus \mathfrak{n}^{+}$. 

\begin{remark}
In the above setting, $\forall \phi \in \mathfrak{h}^{\ast}$, we also denote $\langle \phi, \alpha \rangle = \phi(h_{\alpha})$, $\forall \alpha \in \Phi^{+}$.
\end{remark}

Now we consider the following result (see for instance \cite{Flagvarieties}, \cite{HumphreysLAG}):
\begin{theorem}
Any two Borel subgroups are conjugate.
\end{theorem}
From the result above, given a Borel subgroup $B \subset G^{\mathbbm{C}}$, up to conjugation, we can always suppose that $B = \exp(\mathfrak{b})$. In this setting, given a parabolic Lie subgroup\footnote{A Lie subgroup which contains some Borel subgroup.} $P \subset G^{\mathbbm{C}}$, without loss of generality, we can suppose that
\begin{center}
$P  = P_{I}$, \ for some \ $I \subset \Delta$,
\end{center}
where $P_{I} \subset G^{\mathbbm{C}}$ is the parabolic subgroup which integrates the Lie subalgebra 
\begin{center}

$\mathfrak{p}_{I} = \mathfrak{n}^{+} \oplus \mathfrak{h} \oplus \mathfrak{n}(I)^{-}$, \ with \ $\mathfrak{n}(I)^{-} = \displaystyle \sum_{\alpha \in \langle I \rangle^{-}} \mathfrak{g}_{\alpha}$. 

\end{center}
By definition, we have that $P_{I} = N_{G^{\mathbbm{C}}}(\mathfrak{p}_{I})$, where $N_{G^{\mathbbm{C}}}(\mathfrak{p}_{I})$ is the normalizer in  $G^{\mathbbm{C}}$ of $\mathfrak{p}_{I} \subset \mathfrak{g}^{\mathbbm{C}}$, see for instance \cite[\S 3.1]{Akhiezer}. A complex flag variety $X$ is a compact simply connected homogeneous complex manifold defined by
\begin{equation}
X_{P} := G^{\mathbbm{C}}/P = G/G \cap P,
\end{equation}
where $G^{\mathbbm{C}}$ is a complex simple Lie group with compact real form given by $G$, and $P \subset G^{\mathbbm{C}}$ is a parabolic Lie subgroup. In order to describe the Picard group of $X_{P}$, let us recall some basic facts about the representation theory of $\mathfrak{g}^{\mathbbm{C}}$, a detailed exposition on the subject can be found in \cite{Humphreys}. For every $\alpha \in \Phi$, let
$$\alpha^{\vee} := \frac{2}{\langle \alpha, \alpha \rangle}\alpha.$$ 
The fundamental weights $\{\varpi_{\alpha} \ | \ \alpha \in \Delta\} \subset \mathfrak{h}^{\ast}$ of $(\mathfrak{g}^{\mathbbm{C}},\mathfrak{h})$ are defined by requiring that 
\begin{center}
$\langle \varpi_{\alpha}, \beta^{\vee} \rangle= \begin{cases} 1, \ \ \text{if} \ \ \alpha = \beta,\\
0, \ \ \text{if} \ \ \alpha \neq \beta,\end{cases}$
\end{center}
for every $\alpha, \beta \in \Delta$. We denote by 
\begin{equation}
\Lambda = \bigoplus_{\alpha \in \Delta}\mathbbm{Z}\varpi_{\alpha} \ \ \ \text{and} \ \ \ \Lambda^{+} = \bigoplus_{\alpha \in \Delta}\mathbbm{Z}_{\geq 0}\varpi_{\alpha}, 
\end{equation}
respectively, the set of integral  weights and integral dominant weights of $\mathfrak{g}^{\mathbbm{C}}$. From above, we have that the set of isomorphism classes of finite dimensional representations of $\mathfrak{g}^{\mathbbm{C}}$ is parameterized by $\Lambda^{+}$. In particular, every fundamental weight $\varpi_{\alpha} \in \Lambda^{+}$, $\alpha \in \Delta$, defines a finite dimensional irreducible $\mathfrak{g}^{\mathbbm{C}}$-module $V(\varpi_{\alpha})$ with highest-weight vector $v_{\varpi_{\alpha}}^{+}$. On the other hand, by choosing a trivializing open covering $X_{P} = \bigcup_{i \in J}U_{i}$, in terms of $\check{C}$ech cocycles we can write 
\begin{center}
$G^{\mathbbm{C}} = \Big \{(U_{i})_{i \in J}, \psi_{ij} \colon U_{i} \cap U_{j} \to P \Big \}$.
\end{center}
Given $\varpi_{\alpha} \in \Lambda^{+}$, one can consider the induced character $\vartheta_{\varpi_{\alpha}} \in {\text{Hom}}(T^{\mathbbm{C}},\mathbbm{C}^{\times})$, such that $({\rm{d}}\vartheta_{\varpi_{\alpha}})_{e} = \varpi_{\alpha}$. Since $P = P_{I}$, it can be shown that  
\begin{equation}
{\rm{Hom}}(P_{I},\mathbbm{C}^{\times}) \cong {\rm{Hom}}(T(\Delta \backslash I)^{\mathbbm{C}},\mathbbm{C}^{\times}), \ \ \chi \mapsto \chi|_{T(\Delta \backslash I)^{\mathbbm{C}}},
\end{equation}
such that $T(\Delta \backslash I)^{\mathbbm{C}} \subset T^{\mathbbm{C}}$ is the torus
\begin{equation}
T(\Delta \backslash I)^{\mathbbm{C}} = \exp \Big \{ \displaystyle \sum_{\alpha \in  \Delta \backslash I}a_{\alpha}h_{\alpha} \ \Big | \ a_{\alpha} \in \mathbbm{C} \Big \},
\end{equation}
see for instance \cite[Part II, p. 169]{jantzen2003representations}. From above, for every $\alpha \in \Delta \backslash I$, we have a homomorphism $\vartheta_{\varpi_{\alpha}} \colon P \to \mathbbm{C}^{\times}$ one can equip $\mathbbm{C}$ with the structure of a $P$-space, such that $pz = \vartheta_{\varpi_{\alpha}}(p)^{-1}z$, $\forall p \in P$, and $\forall z \in \mathbbm{C}$. Denoting by $\mathbbm{C}_{-\varpi_{\alpha}}$ this $P$-space, we can form an associated holomorphic line bundle $\mathscr{O}_{\alpha}(1) = G^{\mathbbm{C}} \times_{P}\mathbbm{C}_{-\varpi_{\alpha}}$, which can be described in terms of $\check{C}$ech cocycles by
\begin{equation}
\label{linecocycle}
\mathscr{O}_{\alpha}(1) = \Big \{(U_{i})_{i \in J},\vartheta_{\varpi_{\alpha}}^{-1} \circ \psi_{i j} \colon U_{i} \cap U_{j} \to \mathbbm{C}^{\times} \Big \},
\end{equation}
that is, $\mathscr{O}_{\alpha}(1) = \{g_{ij}\} \in \check{H}^{1}(X_{P},\mathcal{O}_{X_{P}}^{\ast})$, such that $g_{ij} = \vartheta_{\varpi_{\alpha}}^{-1} \circ \psi_{i j}$, $\forall i,j \in J$. 

The following result summarizes the main properties to be considered in this work about invertible coherent sheaves and real $G$-invariant $(1,1)$-forms on flag varieties.
\begin{theorem}
\label{AZADBISWAS}
Given a flag variety $X_{P} = G^{\mathbbm{C}}/P$, such that $P = P_{I}$, for some $I \subset \Delta$, then the following hold:
\begin{enumerate}
\item[(1)] As an abelian group, the Picard group of $X_{P}$ is generated by $\mathscr{O}_{\alpha}(1), \alpha \in \Delta \backslash I$, i.e.,
\begin{equation}
{\rm{Pic}}(X_{P}) = \big \langle \mathscr{O}_{\alpha}(1) \ | \ \alpha \in \Delta \backslash I \big \rangle_{\mathbbm{Z}}.
\end{equation}
\item[(2)] $H^{2}(X_{P},\mathbbm{Z}) = \bigoplus_{\alpha \in \Delta \backslash I}\mathbbm{Z}[{\bf{\Omega}}_{\alpha}]$, such that $c_{1}(\mathscr{O}_{\alpha}(1)) = [{\bf{\Omega}}_{\alpha}], \forall \alpha \in \Delta \backslash I$.
\item[(3)] $\forall \alpha \in \Delta \backslash I$, we have $\pi^{\ast}{\bf{\Omega}}_{\alpha} = \sqrt{-1}\partial \overline{\partial} \varphi_{\varpi_{\alpha}}$, such that $\varphi_{\varpi_{\alpha}} \colon G^{\mathbbm{C}} \to \mathbbm{R}$ is given by
\begin{equation}
\varphi_{\varpi_{\alpha}}(g)  = \log \big (||gv_{\varpi_{\alpha}}^{+}|| \big ), \ \ \forall g \in G^{\mathbbm{C}},
\end{equation}
where $\pi \colon G^{\mathbbm{C}} \to G^{\mathbbm{C}} / P = X_{P}$ is the natural projection.
\item[(4)] The K\"{a}hler cone $\mathcal{K}(X_{P})$ of $X_{P}$ is given explicitly by $\mathcal{K}(X_{P}) = \displaystyle \bigoplus_{\alpha \in \Delta \backslash I} \mathbbm{R}^{+}[ {\bf{\Omega}}_{\alpha}]$.
\end{enumerate}
\end{theorem}
The proof of the above theorem follows from \cite{AZAD}, see also \cite{correa2023deformed}, \cite{FultonWoodward}. 
\begin{remark}
\label{harmonic2forms}
Given any $G$-invariant Riemannian metric $g$ on $X_{P}$, denoting by $\mathscr{H}^{2}(X_{P},g)$ the space of real harmonic 2-forms on $X_{P}$ with respect to $g$, then
\begin{equation}
\mathscr{H}^{2}(X_{P},g) = \mathscr{I}_{G}^{1,1}(X_{P}),
\end{equation}
where $\mathscr{I}_{G}^{1,1}(X_{P})$ is the space of closed $G$-invariant real $(1,1)$-forms on $X_{P}$, for more details, see for instance \cite[Lemma 3.1]{MR528871}.
\end{remark}
In the above setting, we denote the weights of $P = P_{I}$ by  
\begin{center}
$\displaystyle \Lambda_{P} := \bigoplus_{\alpha \in \Delta \backslash I}\mathbbm{Z}\varpi_{\alpha}$. 
\end{center}
From this, the previous theorem provides $\Lambda_{P} \cong {\rm{Hom}}(P,\mathbbm{C}^{\times}) \cong {\rm{Pic}}(X_{P})$, such that\footnote{For every $\alpha \in \Delta \backslash I$, we denote $\mathscr{O}_{\alpha}(\ell) := \mathscr{O}_{\alpha}(1)^{\otimes \ell}$, $\ell \in \mathbbm{Z}$.} 
\begin{enumerate}
\item$ \displaystyle \lambda = \sum_{\alpha \in \Delta \backslash I}k_{\alpha}\varpi_{\alpha} \mapsto \prod_{\alpha \in \Delta \backslash I} \vartheta_{\varpi_{\alpha}}^{k_{\alpha}} \mapsto \bigotimes_{\alpha \in \Delta \backslash I} \mathscr{O}_{\alpha}(k_{\alpha})$.
\item $ \displaystyle {\bf{E}} \mapsto \vartheta_{{\bf{E}}}: = \prod_{\alpha \in \Delta \backslash I} \vartheta_{\varpi_{\alpha}}^{\langle c_{1}({\bf{E}}),[\mathbbm{P}^{1}_{\alpha}] \rangle} \mapsto \phi({\bf{E}}) := \sum_{\alpha \in \Delta \backslash I}\langle c_{1}({\bf{E}}),[\mathbbm{P}^{1}_{\alpha}] \rangle\varpi_{\alpha}$.
\end{enumerate}
Thus, $\forall {\bf{E}} \in {\rm{Pic}}(X_{P})$, we have $\phi({\bf{E}}) \in \Lambda_{P}$. More generally, $\forall \xi \in H^{1,1}(X_{P},\mathbbm{R})$, we can attach $\lambda (\xi) \in \Lambda_{P}\otimes \mathbbm{R}$, such that
\begin{equation}
\label{weightcohomology}
\phi(\xi) := \sum_{\alpha \in \Delta \backslash I}\langle \xi,[\mathbbm{P}^{1}_{\alpha}] \rangle\varpi_{\alpha}.
\end{equation}
From above, for every holomorphic vector bundle ${\bf{E}} \to X_{P}$, we define $\phi({\bf{E}}) \in \Lambda_{P}$, such that 
\begin{equation}
\label{weightholomorphicvec}
\phi({\bf{E}}) := \sum_{\alpha \in \Delta \backslash I} \langle c_{1}({\bf{E}}),[\mathbbm{P}_{\alpha}^{1}] \rangle \varpi_{\alpha},
\end{equation}
where $c_{1}({\bf{E}}) = c_{1}(\bigwedge^{r}{\bf{E}})$, such that $r = \rank({\bf{E}})$.

By the ideas presented in \cite{correa2023deformed}, \cite{FultonWoodward} and \cite{AZAD}), we have the following result.

\begin{lemma}
\label{funddynkinline}
Consider $\mathbbm{P}_{\beta}^{1} = \overline{\exp(\mathfrak{g}_{-\beta}){o}} \subset X_{P}$, such that $\beta \in \Phi_{I}^{+}$. Then, 
\begin{equation}
\int_{\mathbbm{P}_{\beta}^{1}} {\bf{\Omega}}_{\alpha} = \langle \varpi_{\alpha}, \beta^{\vee}  \rangle, \ \forall \alpha \in \Delta \backslash I.
\end{equation}
In particular, the cone of curves ${\rm{NE}}(X_{P})$ is given by ${\rm{NE}}(X_{P}) = \bigoplus_{\alpha \in \Delta \backslash I} \mathbbm{Z}[\mathbbm{P}_{\beta}^{1}]$.
\end{lemma}

\begin{remark}
\label{bigcellcosntruction}
In order to perform some local computations we shall consider the open set $U^{-}(P) \subset X_{P}$ defined by the ``opposite" big cell in $X_{P}$. This open set is a distinguished coordinate neighbourhood $U^{-}(P) \subset X_{P}$ of $o = eP \in X_{P}$ defined as follows
\begin{equation}
\label{bigcell}
 U^{-}(P) =  B^{-}o = R_{u}(P_{I})^{-}o\subset X_{P},  
\end{equation}
 where $B^{-} = \exp(\mathfrak{h} \oplus \mathfrak{n}^{-})$, and
 
 \begin{center}
 
 $R_{u}(P_{I})^{-} = \displaystyle \prod_{\alpha \in \Phi_{I}^{+}}N_{\alpha}^{-}$, \ \ (opposite unipotent radical)
 
 \end{center}
with $N_{\alpha}^{-} = \exp(\mathfrak{g}_{-\alpha})$, $\forall \alpha \in \Phi_{I}^{+}$, e.g. \cite[\S 3]{Lakshmibai2},\cite[\S 3.1]{Akhiezer}. It is worth mentioning that the opposite big cell defines a contractible open dense subset in $X_{P}$, thus the restriction of any vector bundle (principal bundle) over this open set is trivial.
\end{remark}

Combining Theorem \ref{AZADBISWAS}, Lemma \ref{funddynkinline} and Proposition \ref{eigenvalue}, we have the following result.

\begin{proposition}
\label{eigenvalueatorigin}
Let $X_{P}$ be a flag variety and let $\omega$ be a $G$-invariant K\"{a}hler metric on $X_{P}$. Then, for every closed $G$-invariant real $(1,1)$-form $\theta$, the eigenvalues of the endomorphism $\omega^{-1} \circ \theta$ are given by 
\begin{equation}
\label{eigenvalues}
{\bf{\lambda}}_{\beta}(\omega^{-1} \circ \theta) = \frac{ \langle \phi([\theta]), \beta^{\vee} \rangle}{\langle \phi([\omega]), \beta^{\vee} \rangle}, \ \ \beta \in \Phi_{I}^{+},
\end{equation}
such that $\phi([\theta]), \phi([\omega]) \in \Lambda_{P} \otimes \mathbbm{R}$.
\end{proposition}

\begin{proof}
Let $\theta \in \Omega^{1,1}(X_P)^G$ be a closed $G$-invariant real $(1,1)$-form. Since the eigenvalues of the associated endomorphism 
\begin{equation}
\omega^{-1} \circ \theta \colon TX_{P} \to TX_{P}
\end{equation}
are constant it suffices to compute them at the base point $o=eP$. In this particular setting, we can take holomorphic coordinates around $o \in X_{P}$ of the form
\begin{equation}
(w_{\beta})_{\beta \in \Phi_{I}^{+}} \mapsto \exp\Big ( \sum_{\beta \in \Phi_{I}^{+}}w_{\beta}f_{-\beta}\Big )o,
\end{equation}
see Remark \ref{bigcellcosntruction}. Here we consider the isomorphism $\phi_{\beta} \colon \mathfrak{sl}_{2}(\mathbbm{C}) \to \mathfrak{g}_{\beta} \oplus [\mathfrak{g}_{\beta},\mathfrak{g}_{-\beta}] \oplus \mathfrak{g}_{-\beta}$, where 
\begin{equation}
\phi_{\beta} \colon \begin{pmatrix}0 & 1 \\
0 & 0 \end{pmatrix} \mapsto e_{\beta}, \ \  \phi_{\beta} \colon \begin{pmatrix}0 & 0 \\
1 & 0 \end{pmatrix} \mapsto f_{-\beta}, \ \ \phi_{\beta} \colon \begin{pmatrix}1 & \ \ 0 \\
0 & -1 \end{pmatrix} \mapsto \frac{2}{\langle \beta,\beta \rangle}h_{\beta},
\end{equation}
such that $e_{\beta} \in \mathfrak{g}_{\beta}$, $f_{-\beta} \in \mathfrak{g}_{-\beta}$, and $[e_{\beta},f_{-\beta}] = \frac{2}{\langle \beta,\beta \rangle}h_{\beta}$.

In this coordinate system we have $\theta = \sqrt{-1}\partial \bar\partial \varphi$, for some local potential $\varphi$. For each $\beta\in\Phi_I^{+}$, consider the tangent vector 
\begin{center}
$\displaystyle \partial_{w_{\beta}}|_{o}=\left.\frac{d}{dw}\right|_{w=0} \exp(w f_{-\beta})\,\mathrm{o}$.
\end{center}
From above, since each $\partial_{w_{\beta}}|_{o}$ is a eigenvector for the $T^{\mathbb C}\cap G$ action with weight $-\beta$, it follows that $\{\partial_{w_{\beta}}|_{o}\}_{\beta \in \Phi_{I}^{+}}$ diagonalizes any $(T^{\mathbb C}\cap G)$-invariant Hermitian form on $T_{\mathrm{o}}X_P$. Since $\theta$ is $G$-invariant, its associated Hermitian form 
\begin{equation}
\mathcal{H}_{\theta}(X,Y) = -\sqrt{-1}\theta(X,\overline{Y}) = \sum_{\alpha,\beta}\frac{\partial^{2}\varphi}{\partial w_{\alpha}\partial \overline{w_{\beta}}}X^{\alpha}\overline{Y^{\beta}}, 
\end{equation}
such that $X = X^{\alpha}\partial_{w_{\alpha}}$ e $Y = Y^{\beta}\partial_{w_{\beta}}$, is $(T^{\mathbb C}\cap G)$-invariant. Hence, denoting 
\begin{equation}
{\bf{q}}_{\beta}(\theta) := \mathcal{H}_{\theta}\big (\partial_{w_{\beta}},\partial_{\overline{w_{\beta}}} \big)(o) = \frac{\partial^{2}\varphi}{\partial w_{\beta}\partial \overline{w_{\beta}}}(o),
\end{equation}
we obtain the following description
\begin{equation}
\label{eqspec1}
\theta_{\mathrm{o}}=\sum_{\beta\in\Phi_I^{+}}{\bf{q}}_{\beta}(\theta)dw_{\beta}\wedge d \overline{w_{\beta}}\big|_{o}.
\end{equation}

On the other hand, by Eq. (\ref{weightcohomology}) and Lemma \ref{funddynkinline}, it follows that $\theta$ can be written as 
\begin{equation}
\label{eqeig1}
\theta=\sum_{\alpha\in\Delta\backslash I}
\pi c_{\alpha}{\bf{\Omega}}_{\alpha},
\end{equation}
where $c_{\alpha}=\frac{\langle \phi([\theta]), [\mathbbm{P}^{1}_{\alpha}]\rangle}{\pi}$, $\forall \alpha \in \Delta \backslash I$. From the above expression we can show that 
\begin{equation}
\mathcal{H}_{\theta}\big (\partial_{w_{\beta}},\partial_{\overline{w_{\beta}}} \big)(o) = \sum_{\alpha \in \Delta \backslash I} \frac{c_{\alpha}}{2}\langle \varpi_{\alpha},\beta^{\vee} \rangle,
\end{equation}
for all $\beta \in \Phi_{I}^{+}$, see for instance \cite{AZAD}. Thus, we conclude that  
 \begin{equation}
 \label{eqspec2}
 \begin{aligned}
     \theta_{o} &= \frac{\sqrt{-1}}{2}\sum_{\beta \in \Phi^{+}_{I}}\left(\sum_{\alpha \in \Delta \backslash I} \frac{c_{\alpha}}{2}\langle \varpi_{\alpha},\beta^{\vee} \rangle\right) dw_{\beta} \wedge d\bar{w}_{\beta}\big|_{o} \\
     & = \frac{\sqrt{-1}}{4\pi}\sum_{\beta \in \Phi^{+}_{I}} \langle \phi([\theta]),\beta^{\vee} \rangle dw_{\beta} \wedge d\bar{w}_{\beta}\big|_{o}.
\end{aligned}    
 \end{equation}
Then, by comparing the equations \eqref{eqspec1} with \eqref{eqspec2}, we obtain
 \begin{equation}
     {\bf{q}}_{\beta}(\theta)= \frac{\sqrt{-1}}{4\pi} \langle \phi([\theta]),\beta^{\vee} \rangle, \ \ \forall \beta \in \Phi_{I}^{+}.
 \end{equation}
Since $\theta$ is an arbitrary closed $G$-invariant real $(1,1)$-form, we conclude that 
 \begin{equation}
     \lambda_{\beta}(\omega^{-1} \circ \theta) =  \frac{{\bf{q}}_{\beta}(\theta)}{{\bf{q}}_{\beta}(\omega)} = \frac{ \langle \phi([\theta]), \beta^{\vee} \rangle}{\langle \phi([\omega]), \beta^{\vee} \rangle}, \ \ \beta \in \Phi_{I}^{+}.
 \end{equation}
\end{proof}

\begin{remark}
\label{normalizationform}
In the setting of the proof of Proposition \ref{eigenvalueatorigin}, if we consider the change of coordinates
\begin{equation}
z_{\beta} := \sqrt{\frac{\langle \phi([\omega]), \beta^{\vee}\rangle}{2\pi}} w_{\beta}, \ \ \beta \in \Phi_{I}^{+},
\end{equation}
we obtain from the previous result the following description pointwise 
\begin{equation}
 \omega = \sum_{\beta \in \Phi_{I}^{+}} \frac{\sqrt{-1}}{2} dz_{\beta} \wedge d\overline{z_{\beta}} \ \ \ \ {\text{and}} \ \ \ \ \displaystyle \theta = \sum_{\beta \in \Phi_{I}^{+}} \frac{\sqrt{-1}}{2} \lambda_{\beta}(\omega^{-1} \circ \theta)dz_{\beta} \wedge d\overline{z_{\beta}},
\end{equation}
for every closed $G$-invariant real $(1,1)$-form $\theta \in \Omega^{1,1}(X_{P})$.
\end{remark}

\begin{remark}
\label{primitivecalc}
Given ${\bf{\Omega}}_{\alpha} \in c_{1}(\mathscr{O}_{\alpha}(1))$, $\alpha \in \Delta \backslash I$, and fixed some $G$-invariant K\"{a}hler metric $\omega$ on $X_{P}$, since ${\bf{\Omega}}_{\alpha}$ is harmonic with respect to $\omega$ (see Remark \ref{harmonic2forms}), it follows that 
\begin{equation}
-{\rm{d}}^{c}\Lambda_{\omega}({\bf{\Omega}}_{\alpha}) = \delta_{\omega}{\bf{\Omega}}_{\alpha} = 0,
\end{equation}
i.e., $\Lambda_{\omega}({\bf{\Omega}}_{\alpha})$ is constant. Thus, we obtain
\begin{equation}
\label{contraction}
\Lambda_{\omega}({\bf{\Omega}}_{\alpha})= {\rm{tr}}(\omega^{-1} \circ {\bf{\Omega}}_{\alpha}) = \sum_{\beta \in \Phi_{I}^{+}} \frac{\langle \varpi_{\alpha}, \beta^{\vee} \rangle}{\langle \phi([\omega]), \beta^{\vee}\rangle},
\end{equation}
for every $\alpha \in \Delta \backslash I$. In particular, for every holomorphic line bundle ${\bf{E}} \in {\rm{Pic}}(X_{P})$, we have a Hermitian structure ${\bf{h}}$ on ${\bf{E}}$, such that the curvature $F_{\nabla}$ of the associated Chern connection $\nabla \myeq {\rm{d}} + \partial \log ({\bf{h}})$, satisfies 
\begin{equation}
\label{tracecurvature}
\frac{\sqrt{-1}}{2\pi} \Lambda_{\omega}(F_{\nabla}) = \sum_{\beta \in \Phi_{I}^{+} } \frac{\langle \phi({\bf{E}}), \beta^{\vee} \rangle}{\langle \phi([\omega]), \beta^{\vee}\rangle}.
\end{equation}
From this, we have that $\nabla$ is a Hermitian-Yang-Mills (HYM) connection (e.g. \cite{Kobayashi+1987}). Notice that 
\begin{equation}
c_{1}({\bf{E}}) = \sum_{\alpha \in \Delta \backslash I}\langle \phi({\bf{E}}), \alpha^{\vee} \rangle [{\bf{\Omega}}_{\alpha}],
\end{equation}
for every ${\bf{E}} \in {\rm{Pic}}(X_{P})$, i.e., the curvature of the HYM connection $\nabla$ on ${\bf{E}}$ coincides with the $G$-invariant representative of $c_{1}({\bf{E}})$. 
\end{remark}

\subsection{The first Chern class of flag varieties} In this subsection, we shall review some basic facts related with the Ricci form of $G$-invariant K\"{a}hler metrics on flag varieties. Let $X_{P}$ be a complex flag variety associated with some parabolic Lie subgroup $P = P_{I} \subset G^{\mathbbm{C}}$. By considering the identification $T_{{\rm{o}}}^{1,0}X_{P} \cong \mathfrak{m} \subset \mathfrak{g}^{\mathbbm{C}}$, such that 

\begin{center}
$\mathfrak{m} = \displaystyle \sum_{\alpha \in \Phi_{I}^{-}} \mathfrak{g}_{\alpha}$,
\end{center}
 we can realize $T^{1,0}X_{P}$ as being a holomorphic vector bundle, associated with the holomorphic principal $P$-bundle $P \hookrightarrow G^{\mathbbm{C}} \to X_{P}$, such that 

\begin{center}

$T^{1,0}X_{P} = \Big \{(U_{i})_{i \in J}, \underline{{\rm{Ad}}}\circ \psi_{i j} \colon U_{i} \cap U_{j} \to {\rm{GL}}(\mathfrak{m}) \Big \}$,

\end{center}
where $\underline{{\rm{Ad}}} \colon P \to {\rm{GL}}(\mathfrak{m})$ is the isotropy representation. From this, we obtain 
\begin{equation}
\label{canonicalbundleflag}
{\bf{K}}_{X_{P}}^{-1} = \det \big(T^{1,0}X_{P} \big) = \Big \{(U_{i})_{i \in J}, \det (\underline{{\rm{Ad}}}\circ \psi_{i j}) \colon U_{i} \cap U_{j} \to \mathbbm{C}^{\times} \Big \}.
\end{equation}
Since the character $\det \circ \underline{{\rm{Ad}}} \in {\text{Hom}}(P,\mathbbm{C}^{\times})$ is completely determined by its restriction to the torus $T(\Delta \backslash I)^{\mathbbm{C}}$, observing that 
\begin{equation}
\det \underline{{\rm{Ad}}}(\exp({\bf{t}})) = {\rm{e}}^{{\rm{tr}}({\rm{ad}}({\bf{t}})|_{\mathfrak{m}})} = {\rm{e}}^{- \langle \delta_{P},{\bf{t}}\rangle },
\end{equation}
$\forall {\bf{t}} \in {\rm{Lie}}(T(\Delta \backslash I)^{\mathbbm{C}})$, such that $\delta_{P} := \sum_{\alpha \in \Phi_{I}^{+} } \alpha$, and denoting $\vartheta_{\delta_{P}}^{-1} = \det \circ \underline{{\rm{Ad}}}$, it follows that 
\begin{equation}
\label{charactercanonical}
\vartheta_{\delta_{P}} = \displaystyle \prod_{\alpha \in \Delta \backslash I} \vartheta_{\varpi_{\alpha}}^{\langle \delta_{P},\alpha^{\vee} \rangle} \Longrightarrow {\bf{K}}_{X_{P}}^{-1} = \bigotimes_{\alpha \in \Delta \backslash I}\mathscr{O}_{\alpha}(\ell_{\alpha}),
\end{equation}
such that $\ell_{\alpha} = \langle \delta_{P}, \alpha^{\vee} \rangle, \forall \alpha \in \Delta \backslash I$. In particular, notice that 
\begin{equation}
\label{weightcanonical}
\phi({\bf{K}}_{X_{P}}^{-1}) = \delta_{P} = \sum_{\alpha \in \Phi_{I}^{+} } \alpha, 
\end{equation}
see Eq. (\ref{weightholomorphicvec}). If we consider the invariant K\"{a}hler metric $\rho_{0} \in \Omega^{1,1}(X_{P})^{G}$ defined by
\begin{equation}
\label{riccinorm}
\rho_{0} = \sum_{\alpha \in \Delta \backslash I}2 \pi \langle \delta_{P}, \alpha^{\vee} \rangle {\bf{\Omega}}_{\alpha},
\end{equation}
it follows that
\begin{equation}
\label{ChernFlag}
c_{1}(X_{P}) = \Big [ \frac{\rho_{0}}{2\pi}\Big].
\end{equation}
By the uniqueness of $G$-invariant representative of $c_{1}(X_{P})$, we have 
\begin{center}
\label{Ricciinvariant}
${\rm{Ric}}(\omega) = \rho_{0}$, 
\end{center}
for every $G$-invariant Kähler metric $\omega$ on $X_{P}$. In particular, $\rho_{0} \in \Omega^{1,1}(X_{P})^{G}$ defines a $G$-invariant K\"{a}hler-Einstein metric on $X_{P}$ (cf. \cite{MATSUSHIMA}).

\subsection{Borel-Weil-Bott theorem}

Let $\mathscr{W}_{\mathfrak{g}^{\mathbbm{C}}}$ be the Weyl group of $\mathfrak{g}^{\mathbbm{C}}$, i.e., the group generated by the reflections $r_{\alpha} \colon \mathfrak{h}^{\ast} \to \mathfrak{h}^{\ast}$, such that 
\begin{equation}
r_{\alpha}(\phi) = \phi - \langle \phi,\alpha \rangle \alpha^{\vee}, \ \ \phi \in \mathfrak{h}^{\ast},
\end{equation}
for every $\alpha \in \Delta$. As we see from above, the group $\mathscr{W}_{\mathfrak{g}^{\mathbbm{C}}}$ acts naturally on the set of integral weights of $\mathfrak{g}^{\mathbbm{C}}$. This action can be naturally extended to the characters through the characterization
\begin{equation}
\mathscr{W}_{\mathfrak{g}^{\mathbbm{C}}} = N_{G^{\mathbbm{C}}}(T^{\mathbbm{C}})/T^{\mathbbm{C}},
\end{equation}
where $N_{G^{\mathbbm{C}}}(T^{\mathbbm{C}}) \subset G^{\mathbbm{C}}$ is the normalizer of $T^{\mathbbm{C}}$ in $G^{\mathbbm{C}}$. Given $w \in \mathscr{W}_{\mathfrak{g}^{\mathbbm{C}}}$, for the sake of simplicity, we shall denote its representative in $N_{G^{\mathbbm{C}}}(T^{\mathbbm{C}})$ also by $w$.

Considering the weight 
\begin{equation}
\delta^{+} := \frac{1}{2}\sum_{\alpha \Phi^{+}}\alpha = \sum_{\alpha \in \Delta} \varpi_{\alpha}, 
\end{equation}
we have the following definition.
\begin{definition}
Given $\lambda \in \Lambda$, we say that $\lambda$ is a singular weight if $\langle \lambda+ \delta^{+}, \beta^{\vee}\rangle = 0$, for some $\beta \in \Phi^{+}$. We say that $\lambda \in \Lambda$ is a regular weight if it is not singular.
\end{definition}

\begin{remark}
In what follows, given $\lambda \in \Lambda$ and $w \in \mathscr{W}_{\mathfrak{g}^{\mathbbm{C}}}$, we denote 
\begin{equation}
w\star \lambda := w(\lambda + \delta^{+}) - \delta^{+}.
\end{equation}
\end{remark}

For us it will be important the following important result.

\begin{theorem}[Borel-Weil-Bott] 
\label{BWtheorem}
Let $X_{P}$ be a rational homogeneous variety defined by a parabolic Lie subgroup $P \subset G^{\mathbbm{C}}$. Given ${\bf{E}} \in {\rm{Pic}}(X_{P})$, we have the following:
\begin{enumerate}
\item[(i)] If $\phi({\bf{E}})$ is a singular weight, then 
\begin{equation}
H^{q}(X_{P},{\bf{E}}) = \{0\}, \ \ \forall q \geq 0,
\end{equation}
\item[(ii)] If $\phi({\bf{E}})$ is a regular weight, then 
\begin{equation}
H^{q}(X_{P},{\bf{E}})  \cong \begin{cases} V(w\star \phi({\bf{E}}))^{\ast}, \ \ \text{if} \ \ q = \ell(w),\\
\{0\}, \ \ \text{if} \ \ q \neq \ell(w),\end{cases}
\end{equation}
where $\ell(w)$ is the length of the unique element $w \in \mathscr{W}_{\mathfrak{g}^{\mathbbm{C}}}$, such that $w \star \phi({\bf{E}}) \in \Lambda^{+}$.
\end{enumerate}
\end{theorem}

The proof of the above theorem can be found in \cite{serre1954representations}, \cite{bott1957homogeneous}, \cite{demazure1968demonstration}, \cite{demazure1976very}, see also \cite[\S 4.3]{Akhiezer}, \cite[\S 16.4-16.5]{sale2002several}.

\begin{remark}
In the setting of the above theorem, given ${\bf{E}} \in {\rm{Pic}}(X_{P})$, such that $\phi({\bf{E}})$ is a regular weight, it follows from the Weyl dimension formula (e.g. \cite{HumphreysLAG}) that 
\begin{equation}
\label{niceEuler}
\chi(X_{P},{\bf{E}}) = (-1)^{\ell(w)}\dim_{\mathbbm{C}}V(w\star \phi({\bf{E}}))^{\ast}  = (-1)^{\ell(w)}\frac{\Pi_{\alpha \in \Phi^{+}}\langle w(\phi({\bf{E}}) + \delta^{+}), \alpha \rangle}{\Pi_{\alpha \in \Phi^{+}}\langle \delta^{+},\alpha \rangle},
\end{equation}
where $\ell(w)$ is the length of the unique element $w \in \mathscr{W}_{\mathfrak{g}^{\mathbbm{C}}}$, such that $w \star \phi({\bf{E}}) \in \Lambda^{+}$.
\end{remark}

\section{Proof of main results}

In this section, we prove Theorem \ref{theoremA}, Theorem \ref{theoremB} and Theorem \ref{theoremC}.

\subsection{Proof of Theorem A}

\begin{proof}
Given any $x \in X_{P}$, we can take a suitable oriented orthonormal local tangent frame $\{e_{\beta},J(e_{\beta})\}_{\beta \in \Phi_{I}^{+}}$ in an open neighborhood of $x \in X_{P}$, such that 
\begin{equation}
\theta_{x} = \sum_{\beta \in \Phi_{I}^{+}} {\bf{\lambda}}_{\beta}(\omega^{-1} \circ \theta)(x) e_{\beta}^{\ast} \wedge  J(e_{\beta})^{\ast},
\end{equation}
at $x \in X_{P}$, here we consider $dz_{\beta} = e_{\beta}^{\ast} + \sqrt{-1}J(e_{\beta})^{\ast}$, see for instance Remark \ref{normalizationform}. From above, considering $\theta_{x} \colon \mathcal{S}(X_{P})_{x} \to  \mathcal{S}(X_{P})_{x}$, since
\begin{equation}
(e_{\beta}^{\ast} \wedge  J(e_{\beta})^{\ast}) \cdot \psi = e_{\beta} \cdot J(e_{\beta}) \cdot \psi, 
\end{equation}
$\forall \psi \in \mathcal{S}(X_{P})_{x}$ and $\forall \beta \in \Phi_{I}^{+}$, it follows that $\theta_{x}$ acts on $\mathcal{S}(X_{P})_{x}$ via Clifford multiplication as
\begin{equation}
\theta_{x} =  \sum_{\beta \in \Phi_{I}^{+}} {\bf{\lambda}}_{\beta}(\omega^{-1} \circ \theta)(x)e_{\beta} \cdot J(e_{\beta}).
\end{equation}
Since $\theta$ and $\omega$ are $G$-invariant, it follows that 
\begin{equation}
{\bf{\lambda}}_{\beta}(\omega^{-1} \circ \theta)(x) =  \frac{\langle \phi([\theta]),\beta^{\vee} \rangle}{\langle \phi([\omega]),\beta^{\vee} \rangle},
\end{equation}
$\forall \beta \in \Phi_{I}^{+}$ and $\forall x \in X_{P}$. Thus, it is enough to compute the eigenvalues of $\theta \colon \mathcal{S}(X_{P}) \to  \mathcal{S}(X_{P})$ at $o  = eP \in X_{P}$. By choosing an enumeration $\Phi_{I}^{+} = \{ \beta_{1},\ldots,\beta_{m}\}$, and observing that 
\begin{equation}
\mathcal{S}(X_{P})_{o} \cong \underbrace{\mathbbm{C}^{2} \otimes \cdots \otimes \mathbbm{C}^{2}}_{m-\text{times}},
\end{equation}
has a basis of $2^{m}$ spinors $\psi(\epsilon_{\beta_{1}},\ldots,\epsilon_{\beta_{m}})$, such that $m = \dim_{\mathbbm{C}}(X_{P})$ and $\epsilon_{\beta_{j}} = \pm 1$, for every $j = 1,\ldots,m$, and 
\begin{equation}
A_{\beta_{j}} \cdot J(A_{\beta_{j}}) \cdot \psi(\epsilon_{\beta_{1}},\ldots,\epsilon_{\beta_{m}}) = \sqrt{-1}\epsilon_{\beta_{j}}\psi(\epsilon_{\beta_{1}},\ldots,\epsilon_{\beta_{m}}),
\end{equation}
for all $j = 1,\ldots,m$, it follows that the eigenvalues of $\theta_{o} \colon \mathcal{S}(X_{P})_{o} \to  \mathcal{S}(X_{P})_{o}$ are of the form
\begin{equation}
\sqrt{-1}\sum_{\beta \in \Phi_{I}^{+}}\epsilon_{\beta}{\bf{\lambda}}_{\beta}(\omega^{-1} \circ \theta)(o) = \sqrt{-1}\sum_{\beta \in \Phi_{I}^{+}} \epsilon_{\beta}\frac{\langle \phi([\theta]),\beta^{\vee} \rangle}{\langle \phi([\omega]),\beta^{\vee} \rangle},
\end{equation}
such that $\epsilon_{\beta} = \pm 1$, for every $\beta \in \Phi_{I}^{+}$, which concludes the proof.
\end{proof}

\subsection{Proof of Theorem B}

\begin{proof}
(A) Since ${\bf{L}} \in {\rm{Spin}}^{c}(X_{P}) \iff c_{1}({\bf{L}}) = c_{1}(X_{P}) \ ({\rm{mod}} \ 2)$, the result follows from the following facts:
\begin{enumerate}
\item[(i)] From Theorem \ref{AZADBISWAS} and Lemma \ref{funddynkinline}, it follows that 
\begin{equation}
{\bf{L}} = \bigotimes_{\alpha \Delta \backslash I} \mathscr{O}_{\alpha}(1)^{\otimes n_{\alpha}}, \ \ n_{\alpha} = \int_{\mathbbm{P}_{\alpha}^{1}}c_{1}({\bf{L}}), \ \ \forall \alpha \in \alpha \Delta \backslash I.
\end{equation}
\item[(ii)] From Eq. (\ref{weightcanonical}), we have ${\bf{K}}_{X_{P}}^{-1} = \det(T^{1,0}X_{P}) = \bigotimes_{\alpha \Delta \backslash I} \mathscr{O}_{\alpha}(1)^{\otimes \langle \delta_{P},\alpha^{\vee} \rangle}.$
\end{enumerate}
From above, since $c_{1}(X_{P}) = c_{1}({\bf{K}}_{X_{P}}^{-1})$, we conclude that 
\begin{equation}
c_{1}({\bf{L}}) = c_{1}(X_{P}) \ ({\rm{mod}} \ 2) \iff \int_{\mathbbm{P}_{\alpha}^{1}}c_{1}({\bf{L}}) = \langle \delta_{P},\alpha^{\vee} \rangle \ ({\rm{mod}} \ 2), \ \ \forall \alpha \in \Delta \backslash I.
\end{equation}
(B)  If ${\mathcal{D}}_{A} \colon {\mathcal{S}}(X_{P}) \to {\mathcal{S}}(X_{P})$ is the Dirac operator defined by the underlying Yang-Mills connection $A$ on some ${\bf{L}} \in {\rm{Spin}}^{c}(X_{P})$, it follows that 
\begin{equation}
{\mathcal{D}}_{A}^{2} = {\bf{\Delta}}_{A} + \frac{S(\omega)}{4}1_{{\mathcal{S}}(X_{P})} + \frac{F_{A}}{2},
\end{equation}
where $S(\omega)$ is the scalar curvature of the Riemannian metric underlying $\omega$ and $F_{A} = {\rm{d}}A$ is the curvature of the Yang-Mills connection $A$. Now we observe that 
\begin{equation}
S(\omega) = 2 \Lambda_{\omega}({\rm{Ric}}(\omega)) = 4 \pi \sum_{\beta \in \Phi_{I}^{+}} \frac{\langle \delta_{P}, \beta^{\vee} \rangle}{\langle \phi([\omega]), \beta^{\vee}\rangle},
\end{equation}
see for instance Eq. (\ref{tracecurvature}) and Eq. (\ref{Ricciinvariant}). Moreover, since $A$ is  Yang-Mills, it follows that $\theta = \frac{\sqrt{-1}F_{A}}{2\pi}$ is a closed $G$-invariant real $(1,1)$-form. Therefore, it follows that all the eigenvalues of the operator ${\mathcal{D}}_{A}^{2} - {\bf{\Delta}}_{A}$ are constant. From Theorem \ref{theoremA}, we have  
\begin{equation}
{\mathcal{D}}_{A}^{2} - {\bf{\Delta}}_{A} = \Bigg ( \pi \sum_{\beta \in \Phi_{I}^{+}} \frac{\langle \delta_{P}, \beta^{\vee} \rangle}{\langle \phi([\omega]), \beta^{\vee}\rangle}\Bigg) 1_{{\mathcal{S}}(X_{P})} + \frac{\pi}{\sqrt{-1}}\sum_{\beta \in \Phi_{I}^{+}} {\bf{\lambda}}_{\beta}(\omega^{-1} \circ \theta)e_{\beta} \cdot J(e_{\beta}),
\end{equation}
where $\{e_{\beta},J(e_{\beta})\}_{\beta \in \Phi_{I}^{+}}$ is a suitable oriented orthonormal local tangent frame. Since 
\begin{equation}
\phi([\theta]) = \phi(c_{1}(\bf{L})) = \phi({\bf{L}}),
\end{equation}
we conclude that 
\begin{equation}
{\rm{Spec}}\big ( {\mathcal{D}}_{A}^{2} - {\bf{\Delta}}_{A}\big ) = \Bigg \{ \pi  \sum_{\beta \in \Phi_{I}^{+}}\frac{\langle \delta_{P},\beta^{\vee} \rangle + \epsilon_{\beta} \langle \phi({\bf{L}}), \beta^{\vee}\rangle}{\langle \phi([\omega]), \beta^{\vee}\rangle} \ \ \Bigg | \ \ \epsilon_{\beta} = \pm 1, \forall \beta \in \Phi_{I}^{+}\Bigg \}.
\end{equation}
(C) From item (B), one can easily deduce that 
\begin{equation}
\lambda_{\rm{min}} \big ({\mathcal{D}}_{A}^{2} - {\bf{\Delta}}_{A} \big ) =  \pi  \sum_{\beta \in \Phi_{I}^{+}}\frac{\langle \delta_{P},\beta^{\vee} \rangle -  |\langle \phi({\bf{L}}), \beta^{\vee}\rangle|}{\langle \phi([\omega]), \beta^{\vee}\rangle},
\end{equation}
is the smallest eigenvalue of ${\mathcal{D}}_{A}^{2} - {\bf{\Delta}}_{A}$.

(D) In the setting of item (B), under the identification $\mathcal{S}(M) \cong \Lambda^{0,\ast}T^{\ast}M \otimes {\bf{E}}$, we have 
\begin{equation}
\mathcal{D}_{A} = \sqrt{2}\,\big(\bar{\partial}_{{\bf{E}}} + \bar\partial_{{\bf{E}}}^{\ast}\big) \colon \Gamma^{\infty}(\mathcal{S}(X_{P})) \to \Gamma^{\infty}(\mathcal{S}(X_{P})),
\end{equation}  
where ${\bf{E}} = \sqrt{{\bf{L}} \otimes {\bf{K}}_{X_{P}}}$. Therefore, from Borel-Weil-Bott Theorem \ref{BWtheorem}, we obtain
\begin{equation}
\ker(\mathcal{D}_{A}) \;\cong\; \bigoplus_{q=0}^{m} H^{q}(X_{P},{\bf{E}}) = \begin{cases} \{0\}, \ \ \text{if} \ \ \phi({\bf{E}}) \ \ \text{is singular},\\
H^{\ell(w)}(X_{P},{\bf{E}}) \cong V(w\star \phi({\bf{E}}))^{\ast}, \ \ \text{if} \ \ \phi({\bf{E}}) \ \ \text{is regular}.
\end{cases}
\end{equation} 
Hence, the Dirac operator ${\mathcal{D}}_{A} \colon \Gamma^{\infty}({\mathcal{S}}(X_{P})) \to \Gamma^{\infty}({\mathcal{S}}(X_{P}))$ admits a Harmonic spinor if and only if the weight
\begin{equation}
\phi({\bf{E}}) = (1/2)(\phi({\bf{L}}) - \delta_{P}) \in \Lambda_{P},
\end{equation}
is a regular weight. In the case that $\phi({\bf{E}})$ is regular, we have
\begin{equation}
{\rm{Index}}(\mathcal{D}_{A}) = \chi(X_{P},{\bf{E}}) = (-1)^{\ell(w)}\frac{\Pi_{\alpha \in \Phi^{+}}\langle w(\phi({\bf{E}}) + \delta^{+}), \alpha \rangle}{\Pi_{\alpha \in \Phi^{+}}\langle \delta^{+},\alpha \rangle},
\end{equation}
see Remark \ref{niceEuler}, which concludes the proof.
\end{proof}

\subsection{Proof of Theorem C}

\begin{proof}
Suppose that $\lambda$ is an eigenvalue of ${\mathcal{D}}_{A}$, and let $\psi \in \Gamma^{\infty}({\mathcal{S}}(X_{P}))$, such that 
\begin{equation}
{\mathcal{D}}^{2}_{A} \psi = \lambda^{2} \psi.
\end{equation}
By proposition \eqref{l2product} we can consider the $L^{2}$-scalar product 
\begin{equation}
\big < \psi_{1},\psi_{2}\big >:=\int_{X_{P}}(\psi_{1}(x),\psi_{2}(x)){\rm{d}}\mu_{\omega},
\end{equation}
for every $\psi_{1},\psi_{2} \in \Gamma^{\infty}({\mathcal{S}}(X_{P}))$, it follows that
\begin{equation}
\lambda^{2}||\psi||^{2}= \big < {\mathcal{D}}_{A}^{2} \psi,\psi\big > = ||\nabla^{A}\psi||^{2} + \big < ({\mathcal{D}}_{A}^{2} - {\bf{\Delta}}_{A} )\psi,\psi\big >,
\end{equation}
here we have used that $\big < {\bf{\Delta}}_{A}\psi,\psi \big > = \big < \nabla^{A}\psi,\nabla^{A}\psi\big >$. Therefore, since
\begin{equation}
||\nabla^{A}\psi||^{2} +\big < ({\mathcal{D}}_{A}^{2} - {\bf{\Delta}}_{A} )\psi,\psi\big > \geq \lambda_{\rm{min}} \big ({\mathcal{D}}_{A}^{2} - {\bf{\Delta}}_{A} \big ) ||\psi||^{2},
\end{equation}
we conclude that $\lambda^{2} \geq  \lambda_{\rm{min}} \big ({\mathcal{D}}_{A}^{2} - {\bf{\Delta}}_{A} \big )$. Hence, it follows from item (C) of Theorem \ref{theoremB} that
\begin{equation}
\lambda^{2} \geq \pi  \sum_{\beta \in \Phi_{I}^{+}}\frac{\langle \delta_{P},\beta^{\vee} \rangle -  |\langle \phi({\bf{L}}), \beta^{\vee}\rangle|}{\langle \phi([\omega]), \beta^{\vee}\rangle},
\end{equation}
which concludes the proof.
\end{proof}

\bibliographystyle{alpha}
\bibliography{spinc.bib}
\end{document}